\documentclass[12pt, a4paper, reqno]{amsart}

\usepackage{enumerate}
\usepackage{amssymb}
\usepackage{color}
\usepackage{graphics}
\usepackage{epsfig}

\newtheorem{lemma}{Lemma}[section]
\newtheorem{corollary}[lemma]{Corollary}
\newtheorem{theorem}[lemma]{Theorem}
\newtheorem{proposition}[lemma]{Proposition}
\newtheorem{remark}[lemma]{Remark}
\newtheorem{definition}[lemma]{Definition}

\def\Sn{\operatorname{Sn}}
\def\Cs{\operatorname{Cs}}

\newcommand{\R}{\mathbb{R}}
\newcommand{\C}{\mathbb{C}}
\newcommand{\Q}{\mathbb{Q}}
\newcommand{\N}{\mathbb{N}}

\author[J. D. Garc\'{\i}a-Salda\~{n}a]{J. D.
Garc\'{\i}a-Salda\~{n}a$^{(1)}$}

\author[A. Gasull]{A. Gasull$^{(2)}$}

\author[H. Giacomini]{H. Giacomini$^{(3)}$}

\date{}
\dedicatory{} \commby{}

\begin{document}

\title[Bifurcation values ]
{Bifurcation values for a family\\ of planar vector fields of degree five}

\maketitle


\vspace{0.5cm}

\subsection*{Abstract} We study the number of limit cycles and the
bifurcation diagram in the Poincar\'{e} sphere of a one-parameter family of planar
differential equations of degree five $\dot {\bf x}=X_b({\bf x})$ which has been
already considered in previous papers. We prove that there is a value $b^*>0$ such
that the limit cycle exists only when $b\in(0,b^*)$ and that it is unique and
hyperbolic  by using a rational Dulac function. Moreover we provide an interval of
length $27/1000$ where~$b^*$ lies. As far as we know the tools used to determine
this interval are new and are based on the construction of algebraic curves
without contact for the flow of the differential equation. These curves are
obtained using analytic information about the separatrices of the infinite
critical points of the vector field. To prove that the Bendixson-Dulac Theorem
works we develop a method for studying whether one-parameter families of
polynomials in two variables  do not vanish based on the computation of the so
called double discriminant.

\vspace{0.5cm}

\subsection*{Keywords} Polynomial planar system,
 Uniqueness of limit cycles, Bifurcation, Dulac function, Double
discriminant.

\vspace{0.5cm}

\subsection*{2010 Mathematics Subject Classification}{Primary: 34C07; Secondary: 34C23, 34C25, 37C27, 13P15}

\vspace{2cm}

\begin{center}
\small{$^{(1)}$ Departament  de Matem\`{a}tiques,  Universitat Aut\`{o}noma de Barcelona \\
Edifici C. 08193 Bellaterra, Barcelona. Spain.\, email:
\texttt{johanna@mat.uab.cat}\\} \vspace{0.2cm}
\small{$^{(2)}$ {\bf Corresponding author:}\\ Departament de Matem\`{a}tiques,  Universitat Aut\`{o}noma de Barcelona\\
Edifici C. 08193 Bellaterra, Barcelona. Spain.\\ Phone: (34)\, 935812909\, Fax
number: (34)\, 935812790,\, email: \texttt{gasull@mat.uab.cat}\\} \vspace{0.2cm}
\small{$^{(3)}$ Laboratoire de Math\'{e}matiques et Physique Th\'{e}orique.\\ Facult\'{e} des
Sciences et Techniques. Universit\'{e} de Tours. CNRS-UMR 7350.\\ 37200 Tours.
France.\, email: \texttt{Hector.Giacomini@lmpt.univ-tours.fr}\\}
\end{center}

\newpage

\section{Introduction and main results}

Consider the one-parameter family of quintic differential systems
\begin{equation}\label{sisa}
\left\{\begin{array}{lll}
\dot{x}&=&y,\\
\dot{y}&=&-x+(a-x^2)(y+y^3),\quad a\in\mathbb{R}.
\end{array}\right.
\end{equation}
Notice that without the term $y^3$, \eqref{sisa} coincides with the
famous van der Pol system. This family was studied in \cite{Xian}
and the authors concluded that it has only two bifurcation values,
$0$ and $a^*$, and exactly four different global phase portraits on
the Poincar\'e disc. Moreover, they concluded that there exists
$a^*\in(0, \sqrt[3]{9\pi^2/16})\approx(0, 1.77)$, such that  the
system has limit cycles only when $0<a< a^*$ and then if the limit
cycle exists, is unique and hyperbolic. Later, it was pointed out in
\cite{Han-Qian} that the proof of the uniqueness of the limit cycle
had a gap  and a new proof was presented.

System \eqref{sisa} has no periodic orbits when $a\le0$  because in
this case  the function $x^2+y^2$ is a global Lyapunov function.
Thus, from now on, we restrict our attention to the case $a>0$ and
for convenience we write $a=b^{2}$, with $b>0$. That is, we consider
the system
\begin{equation}\label{sisb}
\left\{\begin{array}{lll}
\dot{x}=y,\\
\dot{y}=-x+(b^2-x^2)(y+y^3), \quad b\in\mathbb{R}^+\cup\{0\}.
\end{array}\right.
\end{equation}

Therefore the above family has limit cycles if and only if
$b\in(0,b^*)$ with $b^*=\sqrt{a^*}$ and
$b^*\in(0,\sqrt[6]{9\pi^2/16})\approx(0, 1.33).$ Following
\cite{Xian} we also know that the value $b=0$ corresponds to a Hopf
bifurcation and the value $b^*$ to the disappearance of the limit
cycle in an unbounded polycycle. By using numerical methods it is
not difficult to approach  the value $b^*$. Nevertheless, as far as
we know there are no analytical tools to obtain the value $b^*.$
This is the main goal of this paper.

We have succeed in finding an interval of length $0.027$ containing $b^*$ and
during our study we have also realized that there was a bifurcation value missed
in the previous studies. Our main result is:

\begin{theorem}\label{mteo}
Consider system \eqref{sisb}. Then there exist two positive numbers $\hat b$ and
$b^*$ such that:
\begin{itemize}
\item[$($a$)$]  It has a limit cycle if
and only if $0<b<b^*$. Moreover, when it exists, it is unique, hyperbolic and
stable.
\item[$($b$)$]  The only  bifurcation values of the system are 0,${\hat{b}}$ and
${b^*}$. In consequence there are exactly six different  global phase portraits on
the Poincar\'e disc, which are the ones showed in Figure~\ref{Sphere}.
\item[$($c$)$]  It holds that ${0.79<\hat{b}<b^*<0.817}$.
\begin{figure}[h]

\begin{tabular}{ccc}
\epsfig{file=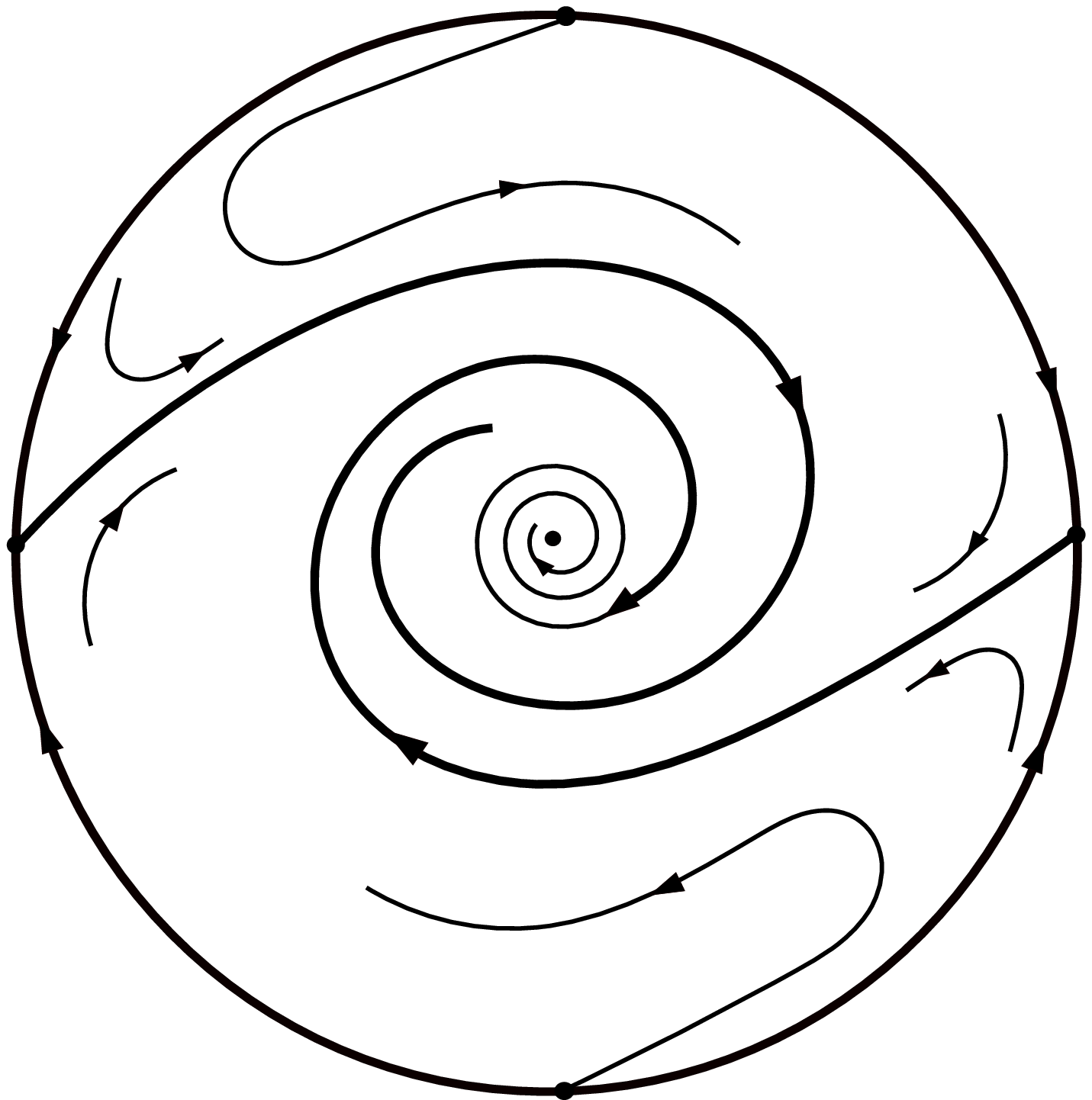,width=4cm,height=4cm} &
\epsfig{file=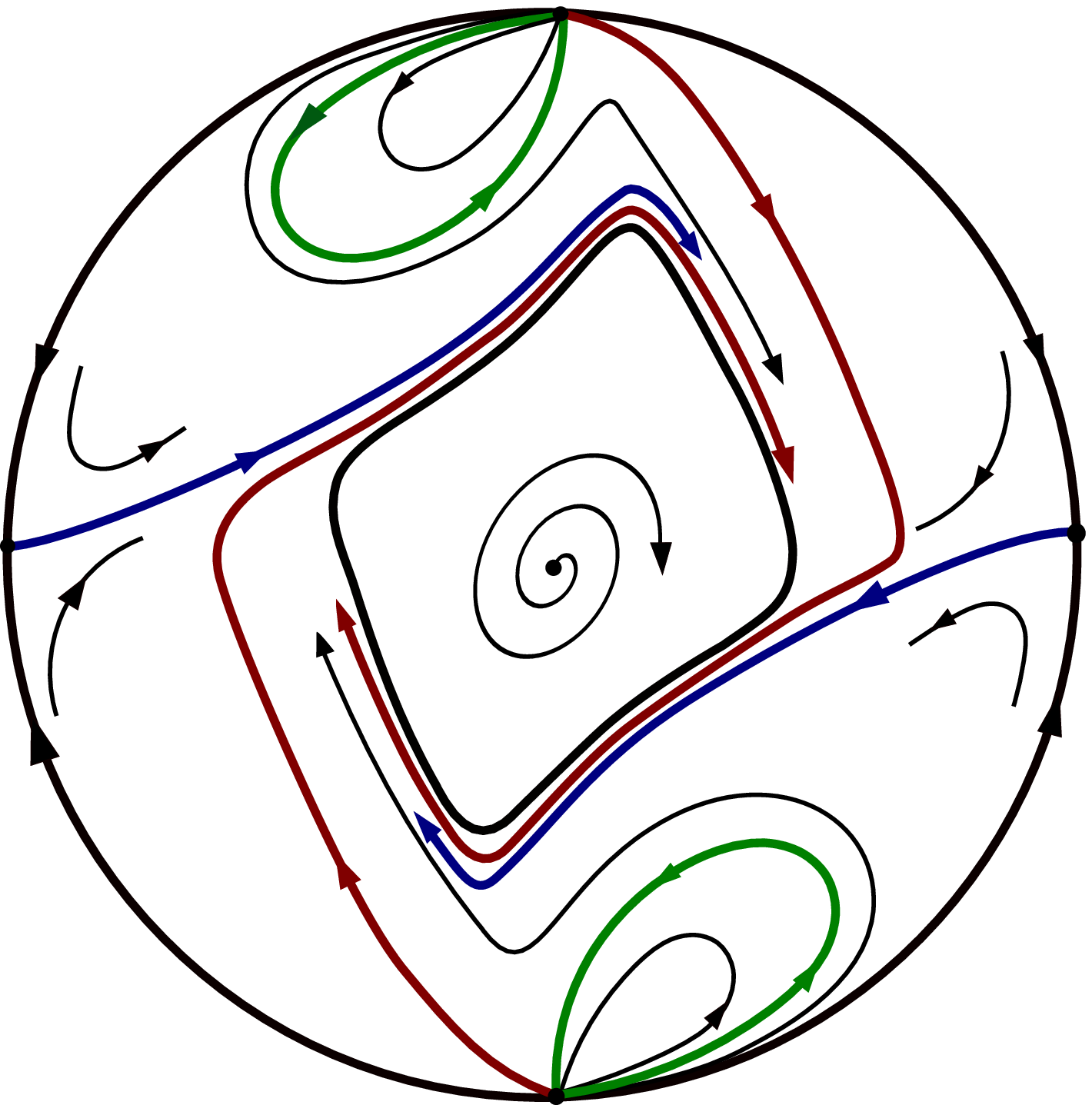,width=4cm,height=4cm}&
\epsfig{file=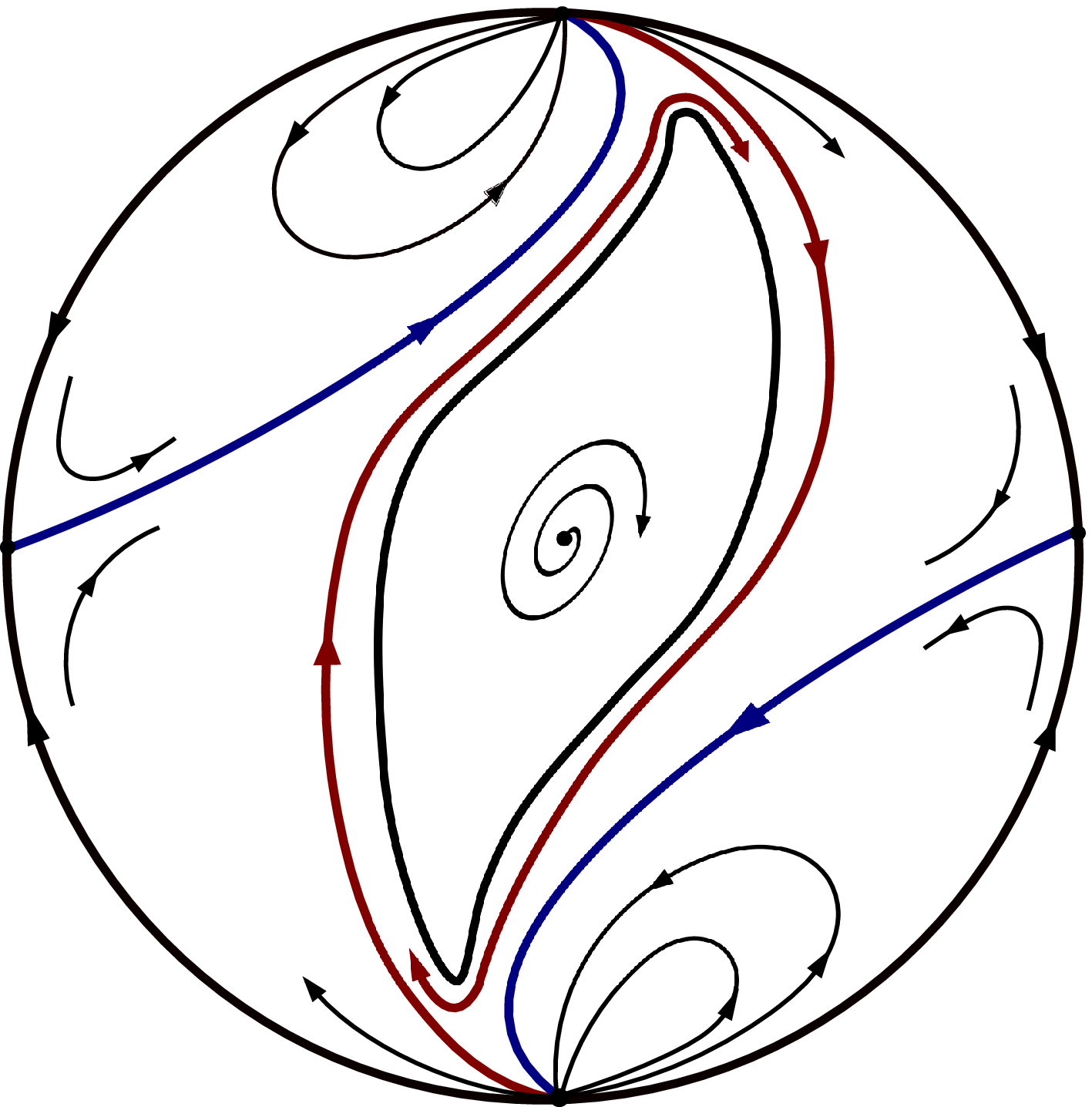,width=4cm,height=4cm}\\
(o) $a\leq 0$& (i) $0<b<\hat{b}$& (ii) $b=\hat{b}$\\&&\\
\epsfig{file=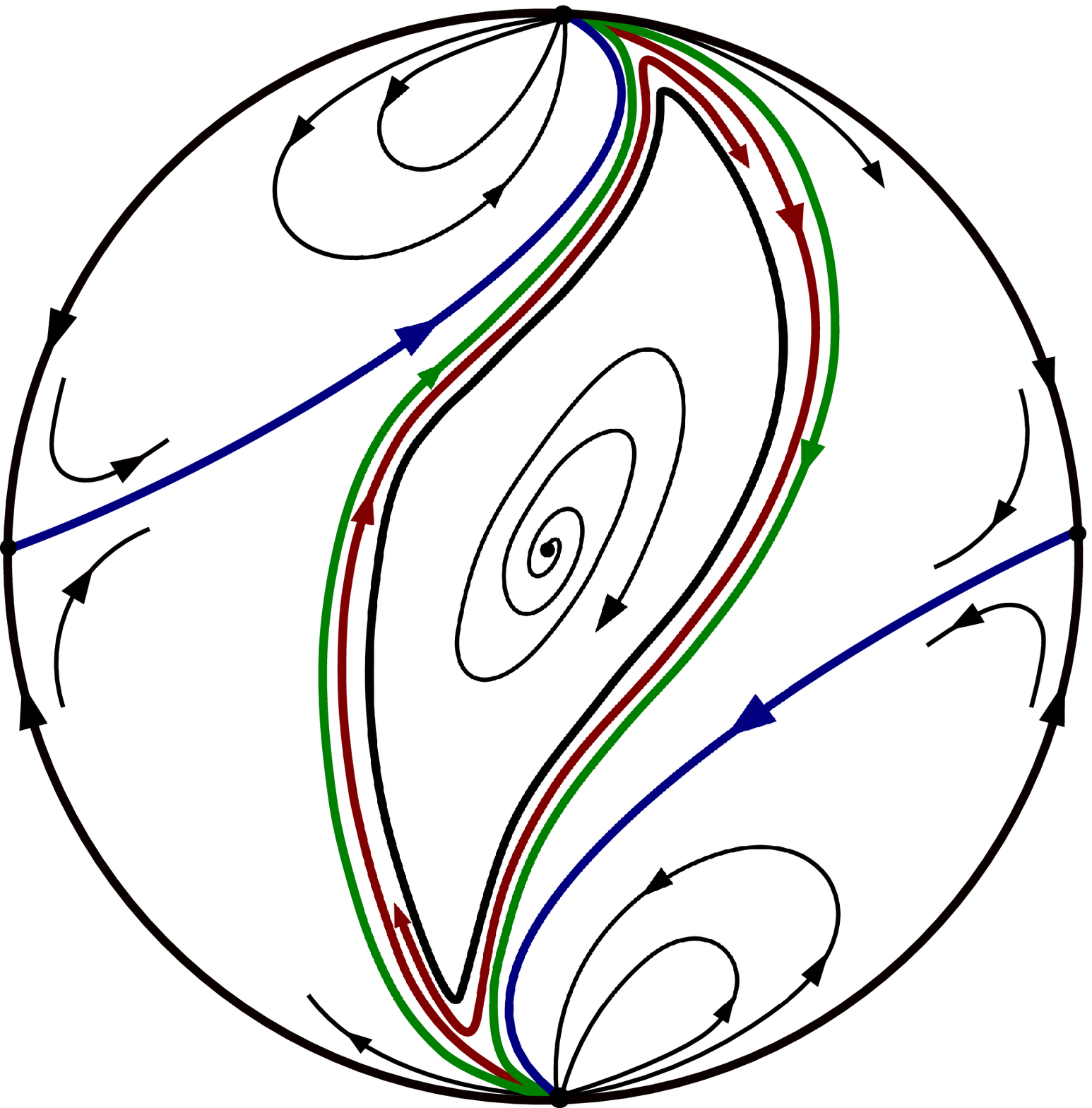,width=4cm,height=4cm}&
\epsfig{file=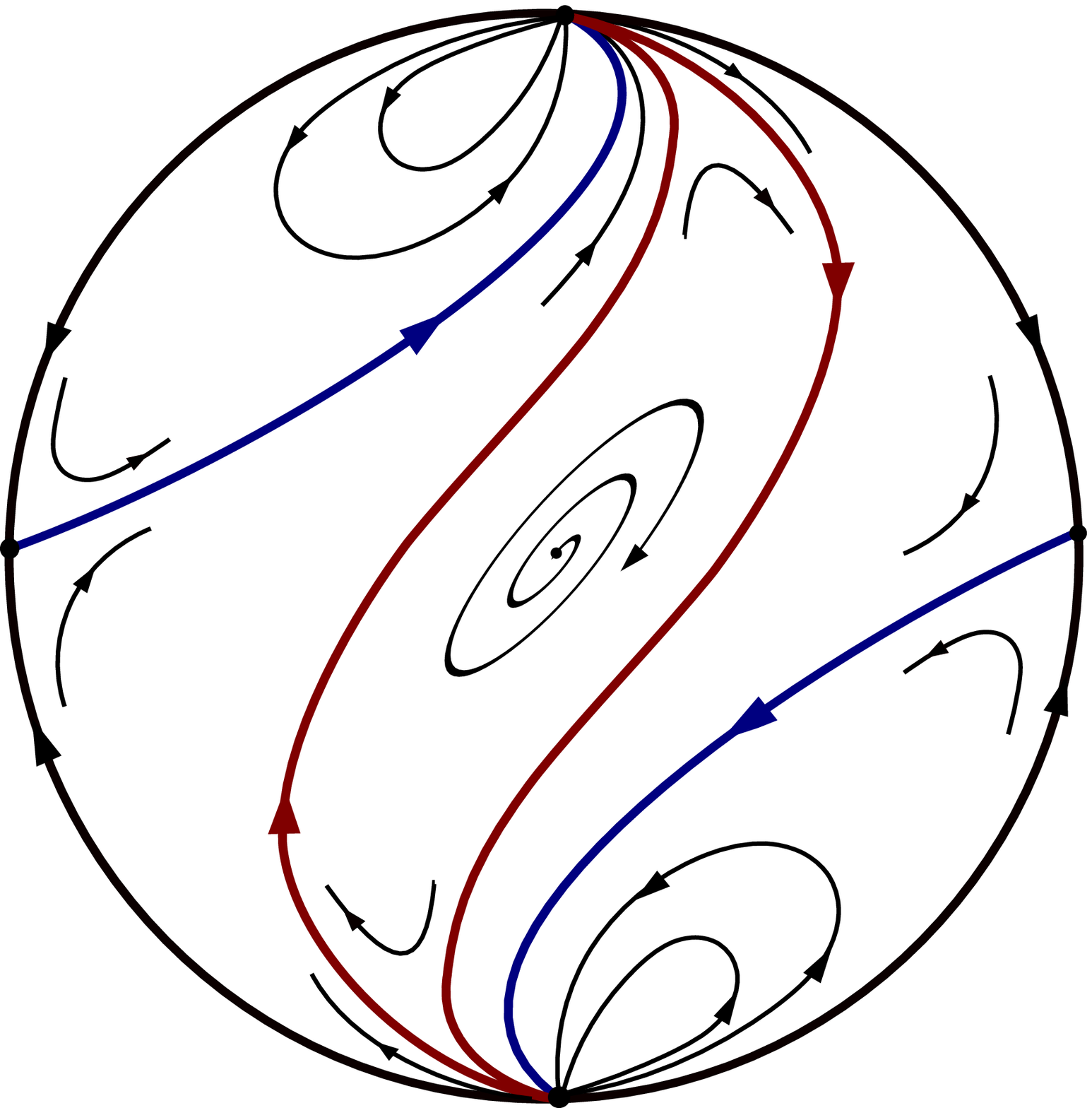,width=4cm,height=4cm}&
\epsfig{file=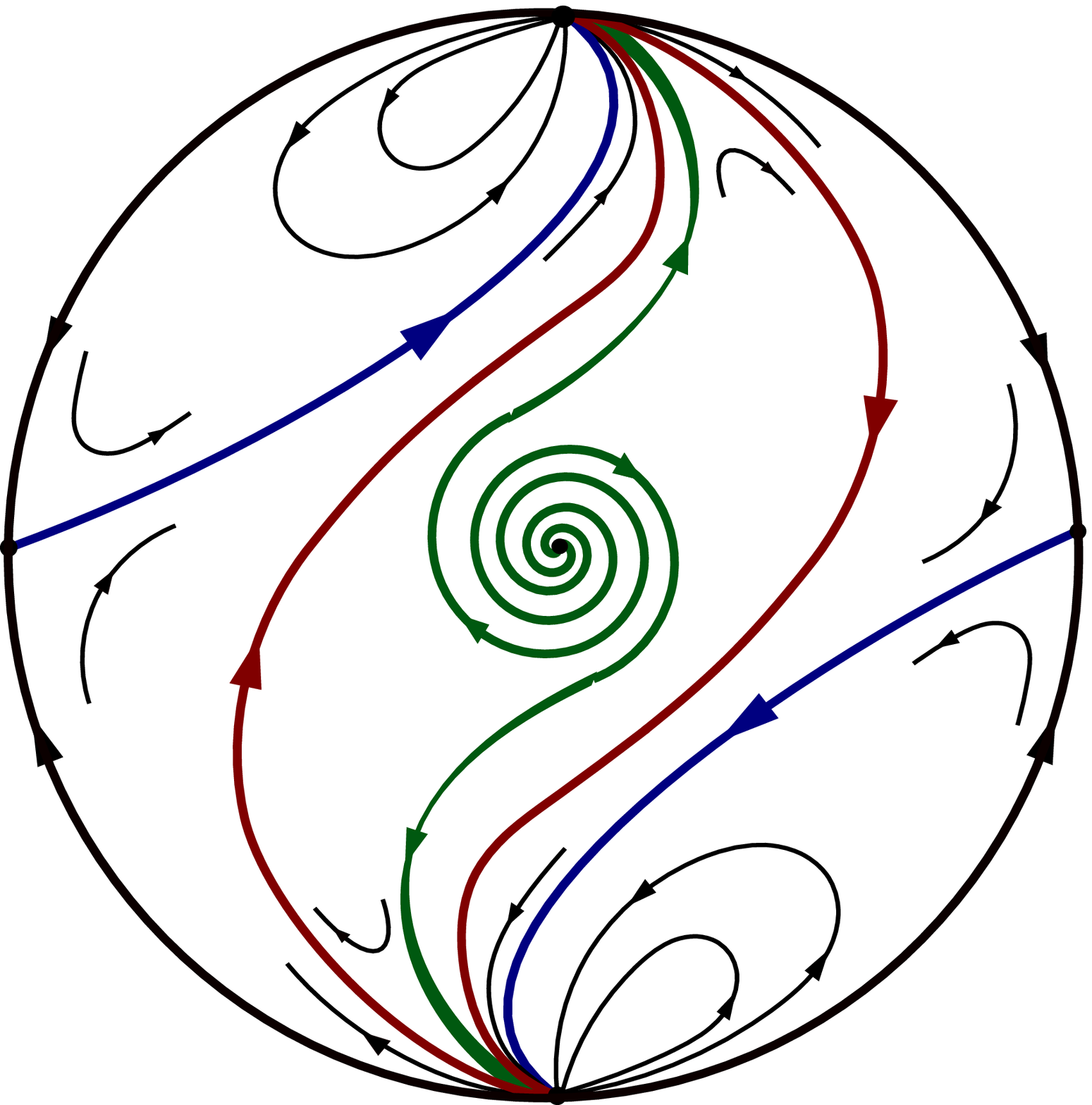,width=4cm,height=4cm}\\
(iii) $\hat{b}<b<b^*$ & (iv) $b=b^*$& (v) $b>b^*$
\end{tabular}
\caption{Phase portraits of systems \eqref{sisa} and  \eqref{sisb}.
When $a\geq0$, then $b=\sqrt{a}.$} \label{Sphere}
\end{figure}
\end{itemize}
\end{theorem}

The phase portraits missed in \cite{Xian} are (ii) and (iii) of Figure \ref{Sphere}.

\smallskip

The key steps in our proof of Theorem~\ref{mteo} are the following:
\begin{itemize}
  \item Give analytic asymptotic expansions of the separatrices of the critical
  points at infinity, see Section~\ref{infin}.
\item  Use these expansions to construct explicit piecewise
rational curves, and prove that they are without contact for the
flow given by ~\eqref{sisb}. These curves allow to control the
global relative positions of the separatrices of the infinite
critical points, see Section~\ref{pp}.
\item Provide an alternative proof of the uniqueness and hyperbolicity
of the limit cycle, which is based in the construction of an explicit  rational
Dulac function, see Section~\ref{uni}.
\end{itemize}

By solving numerically the differential equations we can approach
 the bifurcation values given in the theorem, see Remark~\ref{nume}.
 We have obtained that  $\hat{b}\approx 0.8058459066,$ $b^*\approx 0.8062901027$
  and then  $b^*-\hat b\approx 0.000444$. As we have said the main goal
  of this paper is to get an analytic approach to the more relevant value $b^*,$
  because it corresponds to the disappearance of the limit cycle.

Although all our efforts have been focused on system~\eqref{sisb}, the tools that
we introduce in this work can be applied to other families of polynomial vector
fields and they can provide an analytic control of the bifurcation values for
these families. 

As we will see, our approach is not totally algorithmic and following it we do not
know how to improve the interval presented in Theorem~\ref{mteo} for the values
$\hat b$ and  $b^*$ .

One of the main computational difficulties that we have found has been to prove
that certain polynomials in $x,y$ and $b$, with  high degree, do not vanish on
some given regions. To treat this question, in Appendix II we propose a general
method that uses the so called double discriminant and that we believe that can be
useful in other settings, see for instance \cite{Al-S-Se,Pe-Ro}. In our context
this discriminant turns out to be a huge polynomial in $b^2$ with rational
coefficients. In particular we need to control, on a given interval with rational
extremes, how many reals roots has a polynomial of degree 965,  with enormous
rational coefficients. Although Sturm algorithm theoretically works, in practical
our computers can not deal with this problem using it. Fortunately we can utilize
a kind of bisection procedure based on the Descartes rule (\cite{KM}) to overcome
this difficulty, see Appendix~I.

\smallskip

\section{Structure at infinity}\label{infin}
As usual, for studying the behavior of the solutions at infinity of
system \eqref{sisb} we use the Poincar\'e compactification. That is,
we will use the transformations $(x,y)=\left(1/z,u/z\right)$ and
$(x,y)=\left(v/z,1/z\right)$, with a suitable change of time to
transform system \eqref{sisb} into two new polynomial systems, one
in the $(u,z)$-plane and another one in the $(v,z)$-plane
respectively  (see \cite{An-Le-Go} for details). Then, for
understanding the behavior of the solutions of \eqref{sisb} near
infinity we will study the structure of the critical points of the
transformed systems which are localized on the line $z=0$. Recall
that these points are the {\it critical points at infinity} of
system \eqref{sisb} and their separatrices play a key role for
knowing the bifurcation diagram of the system. In fact, it follows
from the works of Markus~\cite{Ma} and Newmann~\cite{Ne} that it
suffices to know the behavior of these separatrices, the type of
finite critical points and the number and type of periodic orbits to
know the phase portraits of the system.
 We obtain the following result:

\begin{figure}[h]\label{fig2}
\centering\epsfig{file=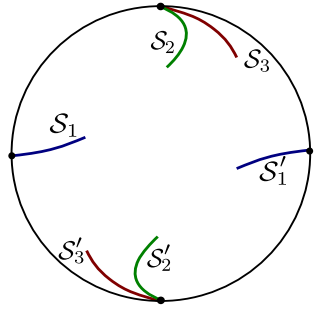,width=4cm,height=4cm}
\caption{Separatrices at infinity for system \eqref{sisb}.}
\end{figure}

\begin{theorem}\label{infi}
System~\eqref{sisb} has six separatrices at infinity, which we
denote by $\mathcal{S}_1,\mathcal{S}_2,$
$\mathcal{S}_3,\mathcal{S}'_1,$ $\mathcal{S}'_2$ and
$\mathcal{S}'_3$, see Figure~\ref{fig2}. Moreover:

\begin{itemize}

\item[$(i)$]Each $\mathcal{S}'_k$ is the image of $\mathcal{S}_k$
under the transformation $(x,y)\rightarrow(-x,-y)$.

\item[$(ii)$]The separatrices $\mathcal{S}_2$ and
$\mathcal{S}_3$ near infinity are contained in the curve
$\{y-\phi(x)=0\}$ where $\phi(x)=\tilde{\phi}(x-b)/(x-b)^2$,
$\tilde{\phi}(u)$ is an analytic function at the origin that
satisfies
\begin{equation}\label{sepS2S3}
\tilde{\phi}(u)=\frac{1}{b}-\frac{1}{3b^2}u+\frac{1}{9b^{3}}u^2-
\frac{359}{27b^4}u^3+O(u^4).
\end{equation}
In particular, $\mathcal{S}_2$ corresponds to $x \lesssim b$ and $\mathcal{S}_3$
to $x \gtrsim b$.

\item[$(iii)$] The separatrix  $\mathcal{S}_1$ near infinity
is contained in the curve $\{y-\varphi(x)=0\}$ where
$\varphi(x)=\tilde{\varphi}(1/x)$ and $\tilde{\varphi}$ is an analytic function at
the origin that satisfies
\begin{equation} \label{sepS1}
\tilde{\varphi}(u)=-u-(b^2-1)u^3- (b^4-3b^2+2)u^5+O\left(u^7\right).
\end{equation}
\end{itemize}
\end{theorem}

\begin{remark} In the statements $(ii)$ and $(iii)$ of Theorem~\ref{infi} the Taylor expansions
of the functions $\tilde{\phi}$ and $\tilde{\varphi}$ can be
obtained up to any given order. In fact, in Section~\ref{pp} we will
use the approximation of $\tilde\phi$ until order 16.
\end{remark}

As a consequence of the above theorem we have the following result:

\begin{figure}[h]
\centering
\begin{tabular}{c}
\begin{tabular}{ccc}
\epsfig{file=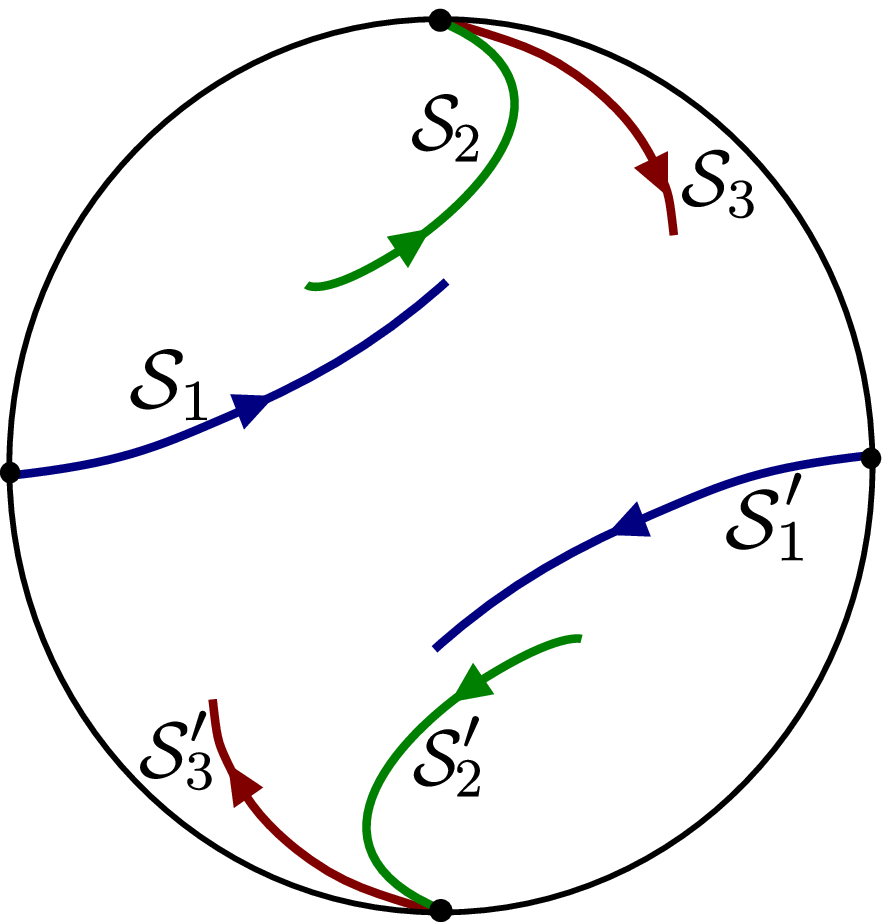,width=3cm,height=3.1cm}&
\epsfig{file=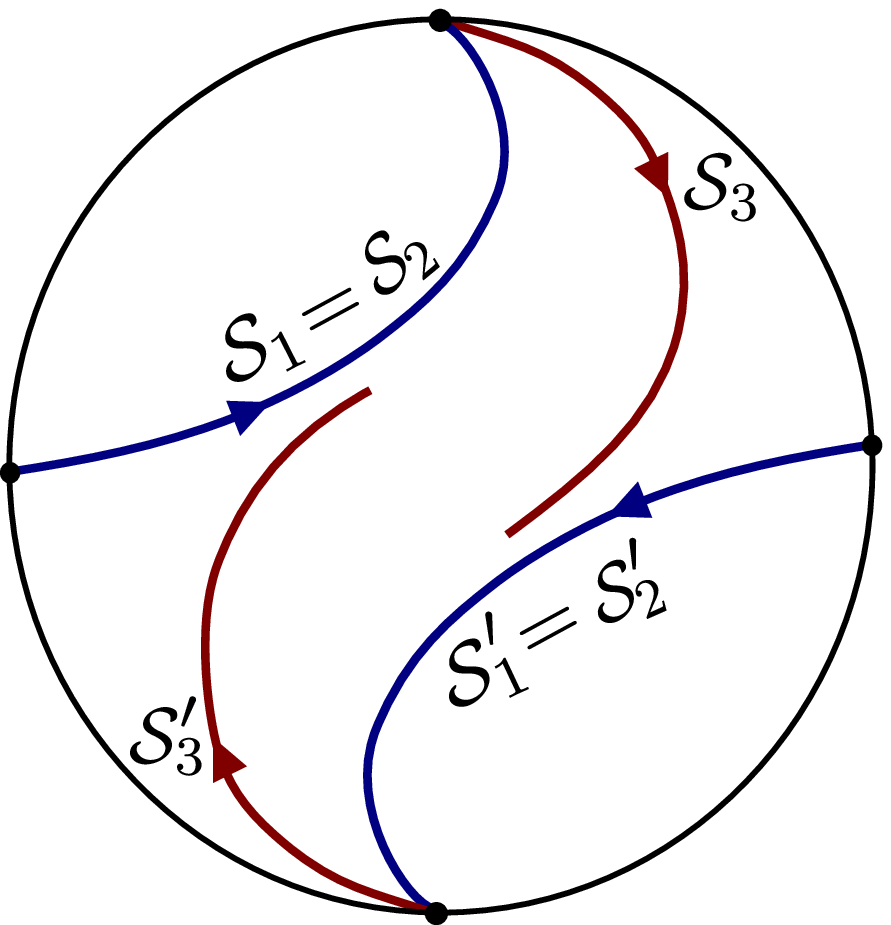,width=3cm,height=3.1cm}&
\epsfig{file=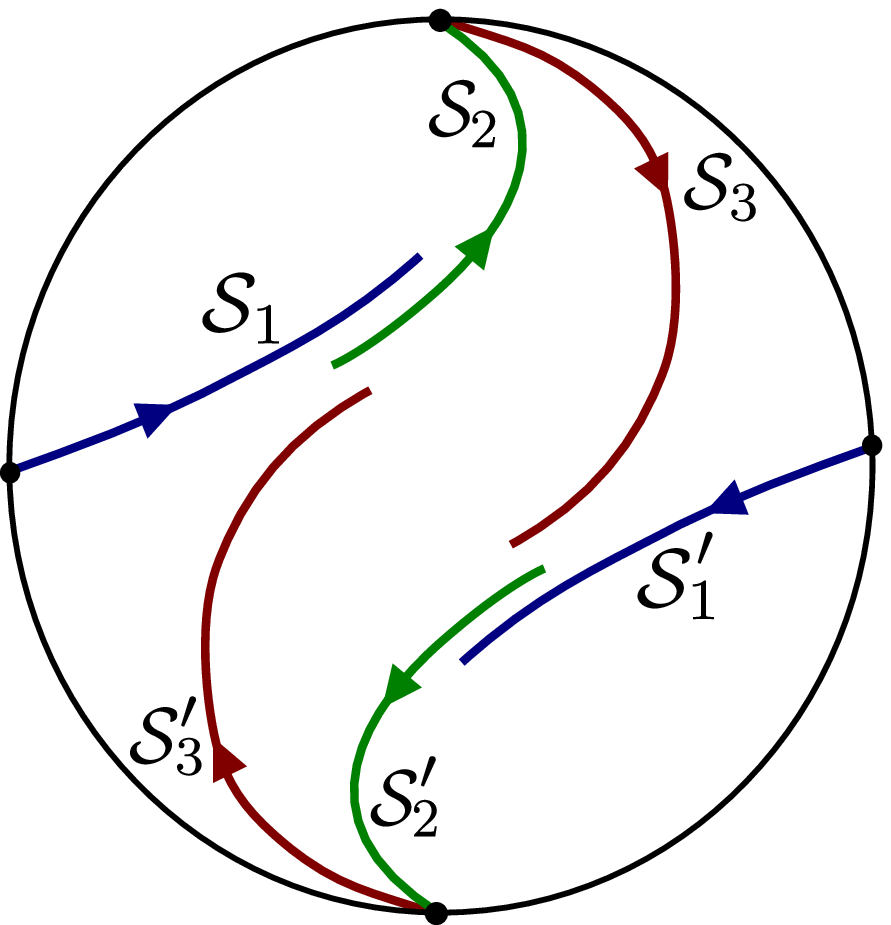,width=3cm,height=3.1cm}\\
$(i)$& $(ii)$& $(iii)$
\end{tabular}
\\
\begin{tabular}{cc}
\epsfig{file=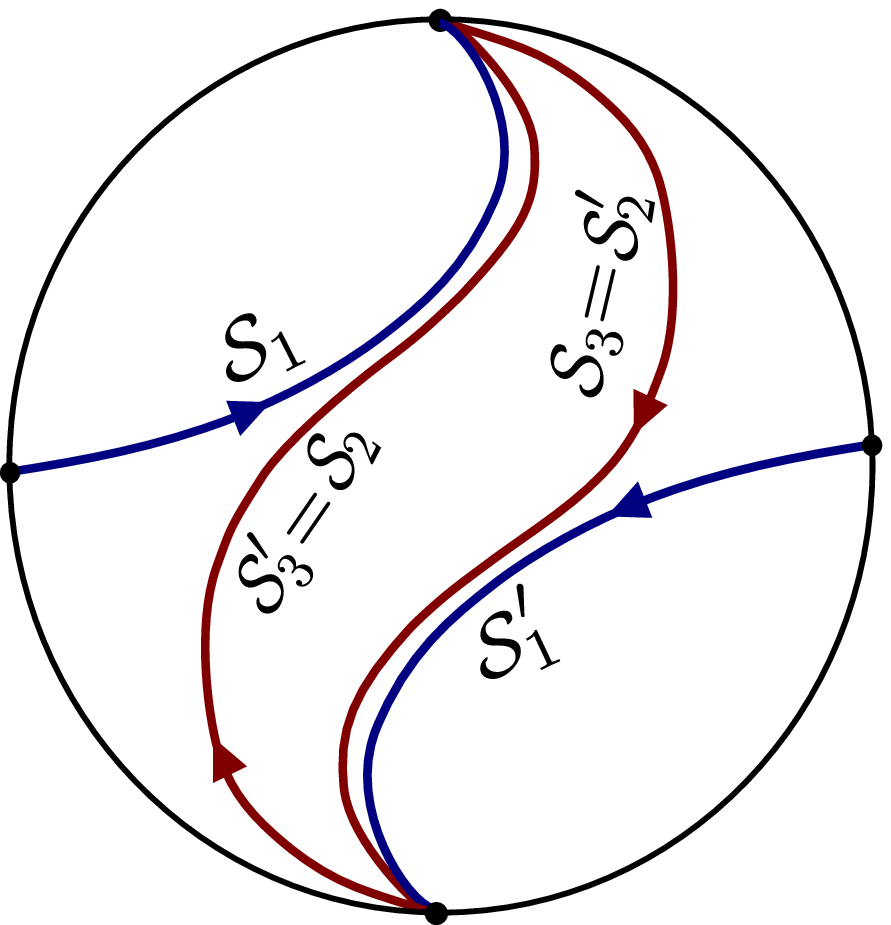,width=3cm,height=3.1cm}&
\epsfig{file=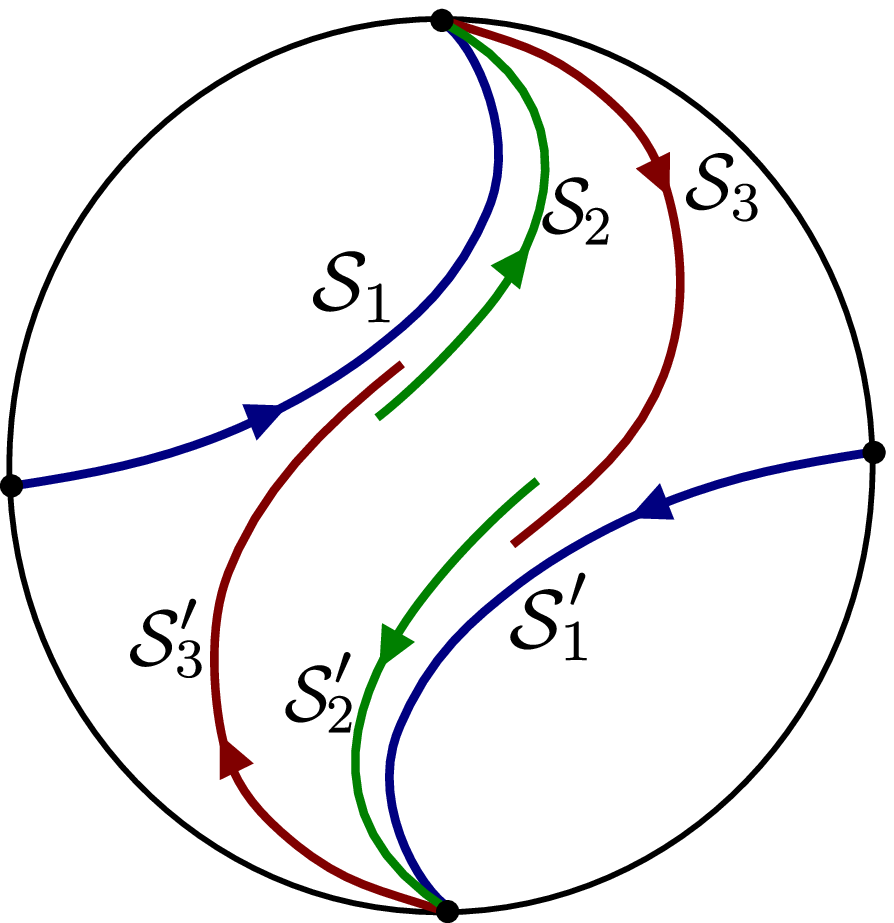,width=3cm,height=3.1cm}\\
$(iv)$& $(v)$
\end{tabular}
\end{tabular}
\caption{Relative position of the separatrices of system \eqref{sisb}.}
\label{posRela}
\end{figure}

\begin{corollary}\label{ccc}
All the possible relative positions  of the separatrices of system \eqref{sisb} in
the Poincar\'{e} disc are given in Figure~\ref{posRela}.
\end{corollary}

To prove the above theorem we need some preliminary lemmas.

\begin{lemma}\label{lem1}
By using the transformation $(x,y)=\left(1/z,u/z\right)$ and the
change of time $dt/d{\tau}=1/{z^4}$ system \eqref{sisb} is
transformed into the system
\begin{equation}\label{CartaUZ}
\left\{\begin{array}{l}
u'=-(1+u^2)z^4-u(1-b^2z^2)(u^2+z^2),\\
z'=-uz^5,
\end{array}\right.
\end{equation}
where the prime denotes the derivative respect to $\tau$. The origin
is the unique critical point of \eqref{CartaUZ} and it is a saddle.
Moreover the stable manifold is the $u$-axis, the unstable manifold,
$\mathcal{S}_1$, is locally contained in the curve
$\{u-\psi(z)=0\}$, where $\psi(z)$ is an analytic function at the origin that satisfies
\begin{equation}\label{unsvar1}
\psi(z)=-z^2-(b^2-1)z^4-(b^4-3b^2+2)z^6+O(z^{8}),
\end{equation}
see Figure~\ref{ff}.
\end{lemma}

\begin{figure}[h]
\centering\epsfig{file=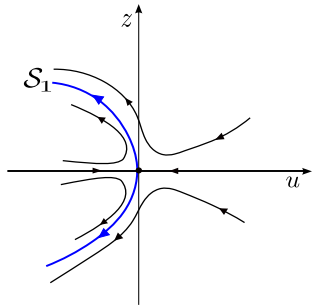,width=4cm,height=4cm} \caption{Phase portrait
of system \eqref{CartaUZ}.}\label{ff}
\end{figure}

\begin{proof}
From the expression of \eqref{CartaUZ} it is clear that the origin
is its unique critical point. For determining its structure we will
use the directional blow-up since the linear part of the system at
this point vanishes identically.

The $u$-directional blow-up is given by the transformation $u=u$,
$q={z}/{u}$; and by using the change of time $dt/d{\tau}={u^2}$,
system \eqref{CartaUZ} becomes
\begin{equation}
\left\{\begin{array}{l}\label{sisvq}
\dot u=-u-(1-b^2u^2)uq^2-(1-b^2u)u^2q^4-u^4q^4,\\
\dot q=q+(1-b^2u^2)q^3+(1-b^2u)uq^5.
\end{array}\right.
\end{equation}
This system has a unique critical point at origin and it is a saddle with
eigenvalues~$\pm 1$.

The  $z$-directional blow-up is given by the transformation
$r={u}/{z}$, $z=z$.  Doing the change of time $dt/d\tau=-{z^2},$
system \eqref{CartaUZ} becomes
\begin{equation}\label{sisrz}
\left\{\begin{array}{l}
\dot r=z+(1-b^2z^2)(r+r^3),\\
\dot z=rz^4.
\end{array}\right.
\end{equation}
This system has a unique critical point at the origin which is
semi-hyperbolic. We will use  the results of \cite[Theorem 65]{An-Le-Go}
to determine its type. By applying the linear change of variables  $r=-\xi+\eta$,
$z=\xi$ system \eqref{sisrz} is transformed into
$$
\left\{\begin{array}{l}
\dot \xi=(\eta-\xi)\xi^4,\\
\dot \eta=\eta-N(\xi,\eta),
\end{array}\right.
$$
where $N(\xi,\eta)=(\eta-\xi)(b^2\xi^2-\xi^4)+(\eta-\xi)^3(b^2\xi^2-1)$. It is
easy to see that if $\eta=n(\xi)$ is the solution of $\eta-N(\xi,\eta)=0$ passing
for the origin, then $n(\xi)=-(b^2-1)\xi^3-(b^4-3b^2+2)\xi^5+O(\xi^7)$. Thus
$(n(\xi)-\xi)\xi^4=-\xi^5+O(\xi^7)$. Therefore from \cite[Theorem 65]{An-Le-Go} we
know that the origin is a semi-hyperbolic saddle. Moreover, its stable manifold is
the $\eta$-axis and its unstable manifold is given by
$$
\eta=-(b^2-1)\xi^3-(b^4-3b^2+2)\xi^5+O(\xi^7).
$$
In the plane $(r,z)$ the local expression of this manifold is
$$
r=-z-(b^2-1)z^3-(b^4-3b^2+2)z^5+O(z^7).
$$

Finally, in the $(u,z)$-plane the unstable manifold is contained in the curve
\eqref{unsvar1} and from the analysis of phase portraits of systems \eqref{sisvq}
and \eqref{sisrz} we obtain that the local phase portrait of system
\eqref{CartaUZ} is the one given in Figure~\ref{ff}.
\end{proof}

\begin{lemma}
By using the transformation $(x,y)=\left(v/z,1/z\right)$ and the
change of time $dt/d{\tau}=1/{z^4}$ system \eqref{sisb} is
transformed into the system
\begin{equation}\label{CartaVZ}
\left\{\begin{array}{lll}
v'&=&v(1+z^2)(v^2-b^2z^2)+(1+v^2)z^4,\\
z'&=&z(1+z^2)(v^2-b^2z^2)+vz^5,
\end{array}\right.
\end{equation}
where the prime denotes the derivative respect to $\tau$. System
\eqref{CartaVZ} has a unique critical point at the origin and
 its local phase portrait is the one showed in Figure \ref{blowup1}.
Moreover, the separatrices $\mathcal{S}_2$ and $\mathcal{S}_3$  are locally
contained in the curve $\{v-g(U)=0\}$ where $U=z/v-1/b$ and $g(U)$ is an analytic
function at the origin that satisfies
\begin{equation}\label{sep}
g(U)=b^6 U^2-\frac{10}{3}b^7U^3+\frac{22}{3}b^8 U^4 +O\left(U^5\right).
\end{equation}
\end{lemma}

\begin{proof}
\begin{figure}[h]
\centering\epsfig{file=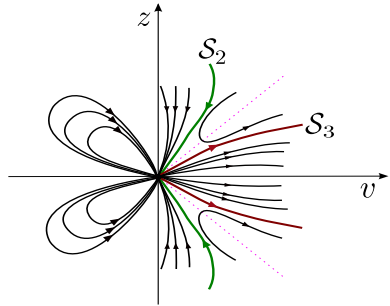,width=5cm,height=4cm}
\caption{Topological local phase portrait of system \eqref{CartaVZ}.
All the solutions are tangent to the $v$-axis but for aesthetical
reasons this fact is not showed in the figure.} \label{blowup1}
\end{figure}

From the expression of system \eqref{CartaVZ} it is clear that the origin is its
unique critical point. As in Lemma \ref{lem1} we will use the directional blow-up
technique to determine its structure since the linear part of the system at this
point is identically zero.

It is well-known, see \cite{An-Le-Go}, that since at the origin
$z'v-v'z=-z^5+O(z^6)$, all the solution, arriving or leaving the
origin have to be tangent to $z=0$. So it suffices to consider the
$v$-directional blow-up  given by the transformation $v=v$, $s=z/v$.
Performing it, together with the change of time $dt/d\tau =-{v^3}$,
system \eqref{CartaVZ} is transformed into
\begin{equation}
\left\{\begin{array}{l}\label{campoVS}
\dot v=-(1+v^2s^2)(1-b^2s^2)-vs^4(1+v^2),\\
\dot s=s^5.
\end{array}\right.
\end{equation}
This system has not critical points. However, by studying the vector field on the
$s$-axis we will obtain relevant information for knowing the phase portrait of
system \eqref{CartaVZ}. If $s=0$ then $\dot{v}=-1$ and $\dot{s}=0$, that is, the
$v$ axis is invariant. If $v=0$ then $\dot{v}=-1+b^2s^2$ and $\dot{s}=s^5$, this
implies that $\dot{v}=0$ if $s=\pm1/b$. In addition, a simple computation shows
that $\ddot{v}>0$ at the points $(0,\pm1/b)$. Therefore the solutions through
these points are as it is showed in Figure~\ref{bw11}.(a), and by the continuity
of solutions with respect to initial conditions, we have that the phase portrait
of system \eqref{CartaVZ}, close to these points, is as it is showed in
Figure~\ref{bw11}.(b).

\begin{figure}[h]
\begin{tabular}{cc}
\hskip 1cm \epsfig{file=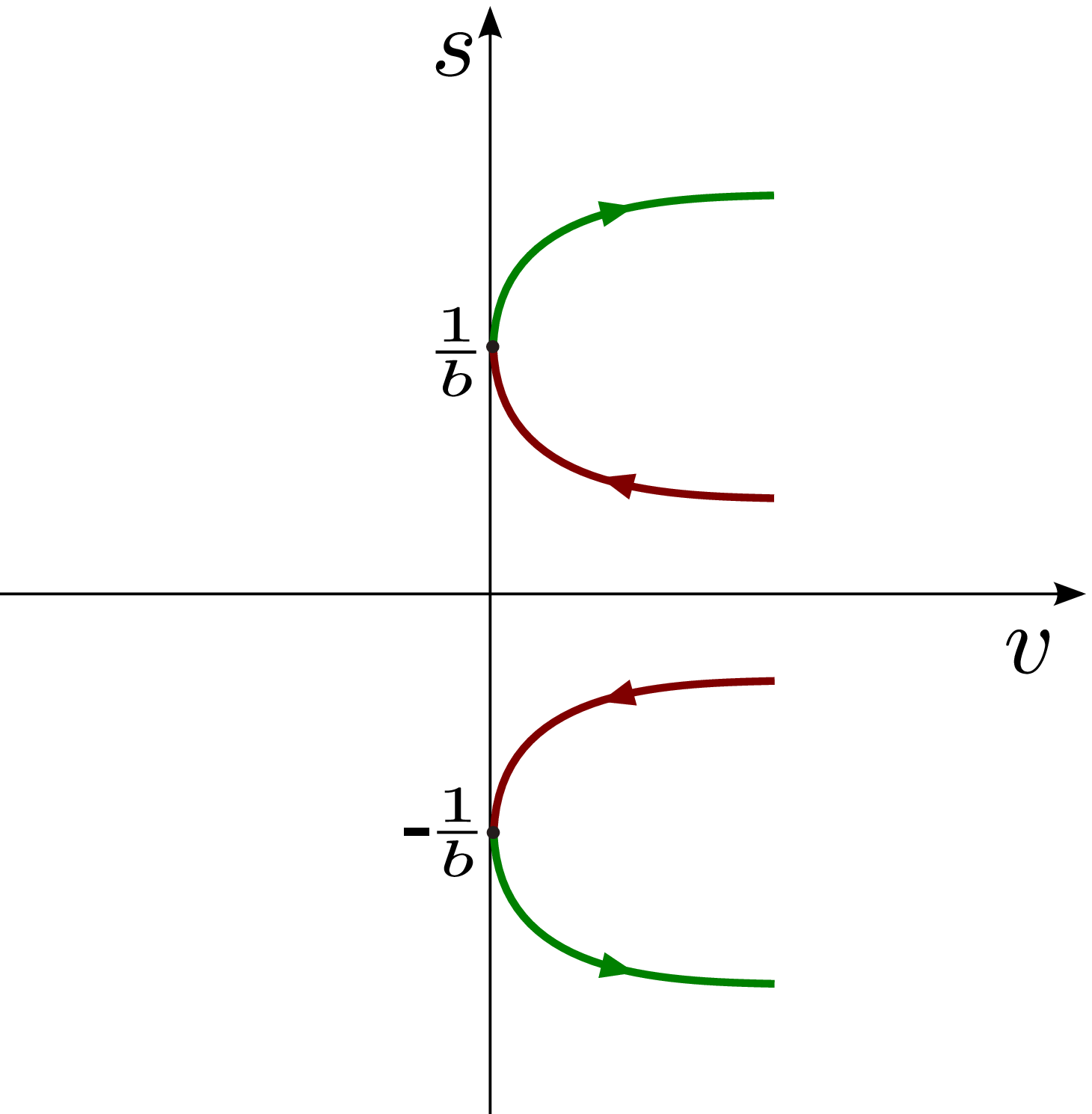,width=3.8cm,height=4cm} &
\hskip 1cm \epsfig{file=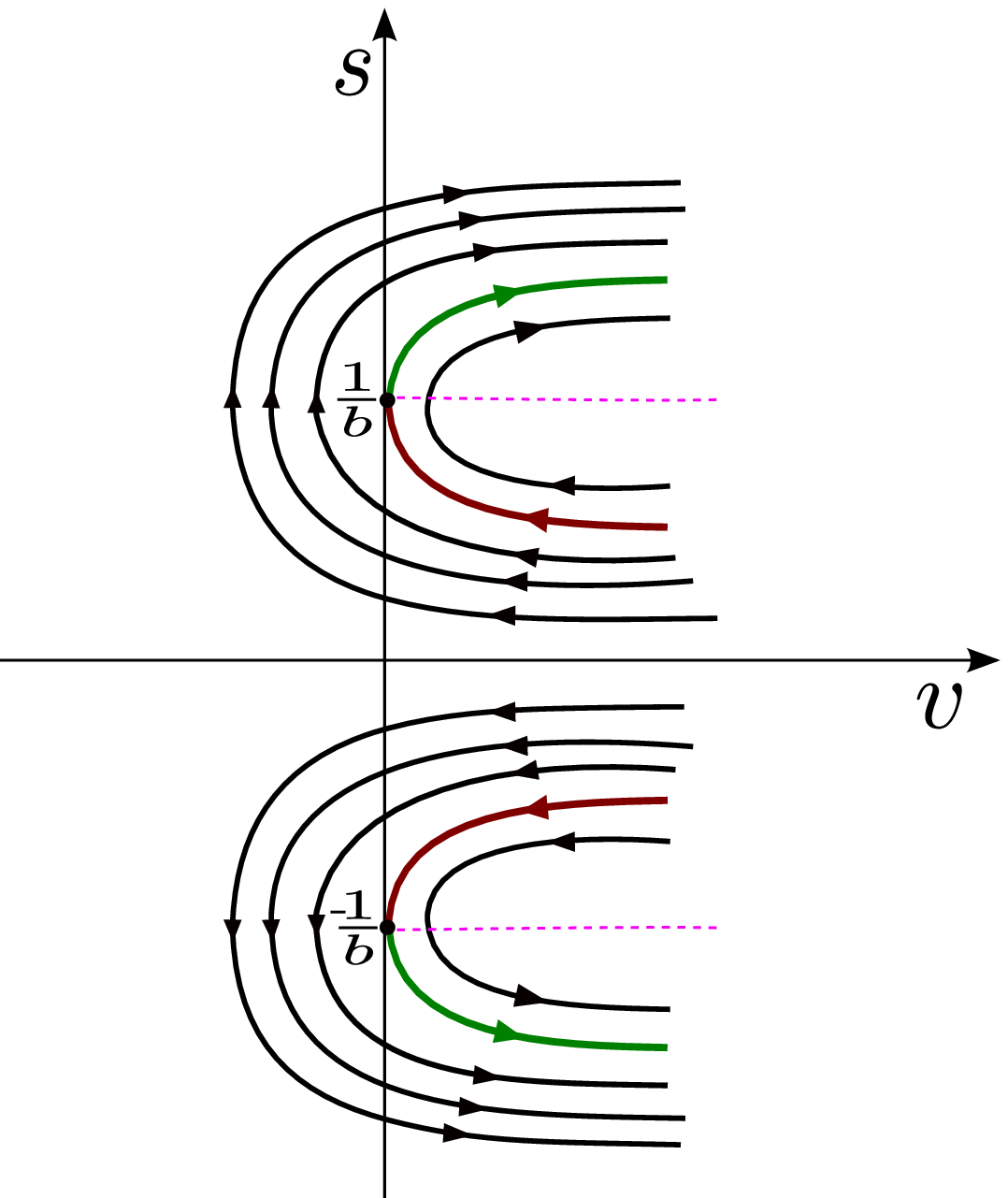,width=3.8cm,height=4cm} \\
\hskip 1cm (a) &\hskip 1cm (b)
\end{tabular}
\caption{Local phase portrait of system~\eqref{campoVS}.}\label{bw11}
\end{figure}

Then by using the transformation $(v,z)=(v,sv)$ and the phase portrait showed in
Figure~\ref{bw11}.(b) we can obtain the phase portrait of system \eqref{CartaVZ}.
Recall that the mapping swaps the second and the third quadrants in the
$v$-directional blow-up. In addition, taking into account the change of time
$dt/d\tau =-{v^3}$ it follows that the vector field in the first and fourth
quadrant of the plane $(v,z)$ has the opposite direction to the showed in the
$(v,s)$-plane. Therefore the local phase portrait of \eqref{CartaVZ} is the showed
in Figure~\ref{blowup1}.

\begin{figure}[h]
\centering\epsfig{file=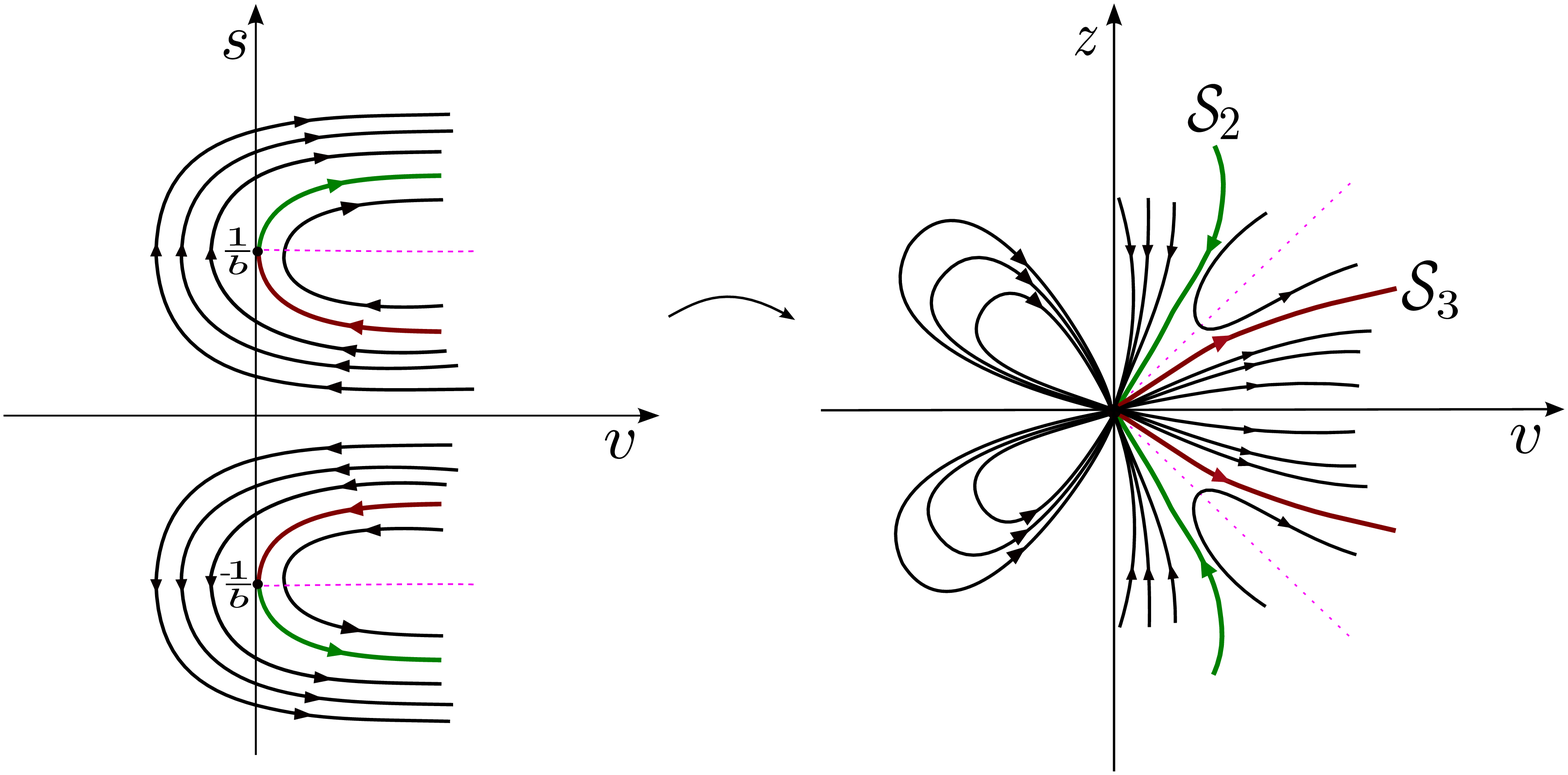,width=7cm,height=4cm}
\caption{Transformation between system~\eqref{campoVS} and
system~\eqref{CartaVZ}.} \label{blowup11}
\end{figure}

To show that the separatrices $\mathcal{S}_2$ and $\mathcal{S}_3$
are contained in the curve \eqref{sep} we proceed as follows. First,
we will obtain the curve that contains the solution through the
point $(0,1/b)$ in the plane $(v,s)$. Second, by using the
transformation $(v,z)=(v,sv)$ we will obtain the corresponding curve
in the $(v,z)$-plane  and we will show that such curve  is exactly
the curve given by \eqref{sep}.

Since $\dot{s}$ is positive in $(0,\infty)$, the solution through
the point $(0,1/b)$ (respectively $(0,-1/b)$) is contained in the
curve $\{v-g(s)=0\}$ (respectively $\{v-\tilde{g}(s)=0\}$), where
$g(s)$ (respectively $\tilde{g}(s)$) is an analytical function
defined in an open neighborhood of the point, moreover it is clear
that $g(1/b)=0$ and $g'(1/b)=0$. Consider the Taylor series of
$g(s)$ around $(1/b)$:
\begin{equation}\label{tayg}
g(s)=\sum_{i=2}^{\infty}\frac{g^{(i)}\left(\frac{1}{b}\right)}{i!}
\left(s-\frac{1}{b}\right)^i.
\end{equation}
Since the curve $\{v-g(s)=0\}$ is invariant then $\langle\nabla
(v-g(s)),\widetilde{X}\rangle=0$ at all the points of $\{v-g(s)=0\}$, where
$\widetilde{X}$ is the vector field associated to system \eqref{campoVS}. Thus, we
have a function, $\langle\nabla (v-g(s)),\widetilde{X}\rangle$, for which all its
coefficients have to be zero. From this observation we obtain  linear recurrent
equations in the coefficients, $g^{(i)}(1/b)$ of $g(s)$. Simple computations show
that the first $3$ terms of the Taylor series of $g(s)$ are:
\begin{center} $
b^6\left(s-\frac{1}{b}\right)^2-\frac{10}{3}b^7\left(s-\frac{1}{b}\right)^3
+\frac{22}{3}b^8\left(s-\frac{1}{b}\right)^4$.
\end{center}
Thus, in the plane $(v,z)$, the curve corresponding to $\{v-g(s)=0\}$ is
\begin{center} $
\left\{v-b^6\left(\frac{z}{v}-\frac{1}{b}\right)^2+
\frac{10}{3}b^7\left(\frac{z}{v}-\frac{1}{b}\right)^3
-\frac{22}{3}b^8\left(\frac{z}{v}-\frac{1}{b}\right)^4
+O\left(\left(\frac{z}{v}-\frac{1}{b}\right)^5\right)=0\right\}$.
\end{center}
Finally, if  $U=z/v-1/b$,  we obtain \eqref{sep}.
\end{proof}

\begin{remark}\label{nume} The proof of the above lemma gives a
natural way for finding a numerical approximation of the value
$b^*$. Notice that in the coordinates $(v,s)$ the point $(0,1/b)$
corresponds to both separatrices $\mathcal{S}_2$ and
$\mathcal{S}_3$. Since it is a regular point we can start our
numerical method (we use a Taylor method) without initial errors and
then  follow the flow of the system, both forward and backward for
given fixed times, say $t^+>0$ and $t^-<0$. We arrive to the points
$(v^\pm,s^\pm)$ with $s^\pm\neq 0$ for $t=t^\pm$, respectively.
These two points have associated two different points
$(x^\pm,y^\pm)$  in the plane $(x,y)$, because of the transformation
$(v,s)=(x/y,1/x)$. Now, we integrate numerically the system
\eqref{sisb} with initial conditions $(x^\pm,y^\pm)$  to continue
obtaining approximations of the separatrices $\mathcal{S}_2$ and
$\mathcal{S}_3$, respectively. The next step is to compare the
points of intersection $\tilde{x}^+=\tilde{x}^+(b)<0$ and
$\tilde{x}^-=\tilde{x}^-(b)>0$ of these approximations with the
$x$-axis.

We consider the function $b\to \Pi(b):=x^+(b)+\tilde{x}^-(b)$ and we
use the bisection method to find one approximate zero of\,  $\Pi$.
Note that if\, $\Pi(\bar b)=0$ then $\mathcal{S}_2'=\mathcal{S}_3$
and by the symmetry of the system $\mathcal{S}_3'=\mathcal{S}_2,$
and therefore $b^*=\bar b$. Taking $b_0=0.8062901027$, $t^+=0.05$
and $t^-=-0.5$ we obtain that
$\tilde{x}^+(b_0)+\tilde{x}^-(b_0)\approx -4.58036036\times10^{-11}$
and so $b^*\approx b_0$.

Following a similar procedure, but now using Lemma~\ref{lem1} to have an initial
condition almost on $\mathcal{S}_1$, we get that $\hat{b}\approx 0.8058459066$.
\end{remark}

\begin{proof}[Proof of Theorem~\ref{infi}] $($i$)$
The result follows because system~\eqref{sisb} is invariant by the transformation
$(x,y)\rightarrow(-x,-y)$.

$($ii$)$. From \eqref{sep} and by using the change of variables $(v,z)=(x/y,1/y)$
we obtain that the separatrices $\mathcal{S}_2$ and $\mathcal{S}_3$ are contained
in the curve {\footnotesize
$$
\left\{\frac{x}{y}-b^6\left(\frac{1}{x}-\frac{1}{b}\right)^2+
\frac{10}{3}b^7\left(\frac{1}{x}-\frac{1}{b}\right)^3
-\frac{22}{3}b^8\left(\frac{1}{x}-\frac{1}{b}\right)^4
+O\left(\left(\frac{1}{x}-\frac{1}{b}\right)^5\right)=0\right\},
$$
} or equivalently \begin{equation}\label{equ1}
\left\{y-\overline{\phi}(x)=0\right\},
\end{equation}
where
$$
\overline{\phi}(x)=\frac{x}{b^6\left(\frac{1}{x}-\frac{1}{b}\right)^2-
\frac{10}{3}b^7\left(\frac{1}{x}-\frac{1}{b}\right)^3
+\frac{22}{3}b^8\left(\frac{1}{x}-\frac{1}{b}\right)^4
+O\left(\left(\frac{1}{x}-\frac{1}{b}\right)^5\right)}.
$$
We can write the function $\overline{\phi}(x)$ as
\begin{equation}\label{equ2}
\overline{\phi}(x)=\left(\frac{1}{(x-b)^2}\right)\overline{\phi}_1(x),
\end{equation}
where
$$
\overline{\phi}_1(x)=\frac{b^2x^3}{b^6+\frac{10}{3}b^7\left(\frac{x-b}{bx}\right)
+\frac{22}{3}b^8\left(\frac{x-b}{bx}\right)^2
+O\left(\left(\frac{x-b}{bx}\right)^3\right)}.
$$
The function $\overline{\phi}_1(x)$ is analytical at $x=b$ and it is
not difficult to see that it has the following Taylor expansion
$$
\overline{\phi}_1(x)=\frac{1}{b}-\frac{(x-b)}{3b^2}+\frac{(x-b)^2}{9b^3}-
\frac{359(x-b)^3}{27b^4}+O((x-b)^4).
$$
Then \eqref{equ2} can be written as
$$
\overline{\phi}(x)=\frac{1}{b(x-b)^2}-\frac{1}{3b^2(x-b)}+\frac{1}{9b^3}-
\frac{359}{27b^4}(x-b)+O((x-b)^2).
$$
Hence from \eqref{equ1} and taking $\overline{\phi}(x)=\tilde{\phi}(x-b)/(x-b)^2$
we complete the proof.

The proof of $($iii$)$  follows by applying the previous ideas, considering the
expression given by \eqref{unsvar1} and the change of variables $(u,z)=(y/x,1/x)$.
\end{proof}

\section{Proof of Theorem~\ref{mteo}}

We start proving a preliminary result that is a consequence of some general
properties of semi-complete family of rotated vector fields with respect one
parameter, SCFRVF for short, see \cite{Duff,Perko1}.

\begin{proposition}\label{fr} Consider system~\eqref{sisb} and assume that for $b=\bar b>0$ it
has no limit cycles. Then there exists  $0<b^*\le \bar b$ such that the system has
limit cycles if and only if $b\in(0,b^*)$.  Moreover, for $b=b^*$ its phase
portrait is like~(iv) in Theorem~\ref{mteo} and when $b>b^*$ it is like~(v) in
Theorem~\ref{mteo}.

\end{proposition}

\begin{proof} It is easy to see that the system has a limit cycle for $b\gtrsim
0$, which appears from the origin through an Andronov-Hopf bifurcation.

If we denote by $X_b(x,y)=(P_b(x,y), Q_b(x,y))$ the vector field
associated to (\ref{sisb}) then
\begin{align*}
\frac{\partial}{\partial b^2} \arctan \left( \frac
{Q_b(x,y)}{P_b(x,y)} \right) &=\frac{P_b(x,y)\frac{\partial
Q_b(x,y)}{\partial b^2}-Q_b(x,y)\frac{\partial P_b(x,y)}{\partial
b^2} }{P^2_b(x,y)+Q^2_b(x,y)} \\&=\frac{
y^2(1+y^2)}{P^2_b(x,y)+Q^2_b(x,y)}\ge0.
\end{align*}
This means that system (\ref{sisb}) is a  SCFRVF  with respect to
the parameter $b^2$.

We will recall  two  properties of SCFRVF. The first one is the so
called {\it non-intersection property.} It asserts that if
$\gamma_{1}$ and $\gamma_{2}$ are limit cycles corresponding to
different values of $b,$ then $\gamma_{1}\cap \gamma_{2}=\emptyset.$

The second one is called {\it planar termination principle}:
\cite{Perko2, Perko3} if varying the parameter we follow with
continuity a limit cycle generated from a critical point ${\bf p},$
we get that the union of all the limit cycles covers a 1-connected
open set $\mathcal{U},$ whose boundaries are ${\bf p}$ and a cycle
of separatrices of $X_b.$ The corners of this cycle of separatrices
are finite or infinite critical points of $X_b.$ Since in our case
$X_b$ only has  the origin as a finite critical point we get that
$\mathcal{U}$ has to be unbounded. Notice that in this definition,
when a limit cycle goes to a semistable limit cycle then we continue
the other limit cycle that has collided with it. This limit cycle
has to exist, again by the properties of SCFRVF.

If for some value of $b=\bar b>0$ the system has no limit cycle it
means that the limit cycle starting at the origin for $b=0,$ has
disappeared for some $b^*,$
 $0<b^*\le \bar b$  covering the whole set $\mathcal{U}.$ Since $\mathcal{U}$ fills
 from  the
origin until infinity, from the non intersection property, the limit cycle cannot
either exist for $b\ge b^*,$ as we wanted to prove.

Since for $b>0$ the origin is a repellor, by Corollary~\ref{ccc} we know by the
Poincar\'{e}-Bendixson Theorem that the phase portraits (i),(ii) and (iii) in
Figure~\ref{Sphere} have at least one limit cycle. Then,  the phase portraits for
$b\ge b^*$ have to be like (iv) or (v) in the same figure. Since the phase
portrait~(iv)  is the only one having a cycle of separatrices it corresponds to
$b=b^*$. Again by the properties of SCFRVF, the phase portrait (iv) does not
appear again for $b>b^*$. Hence, for $b>b^*$ the phase portrait has to be like~(v)
and the proposition follows.
\end{proof}

\begin{remark} In Lemma \ref{b=1} we will give a simple proof that when $b=1$
system~\eqref{sisb} has no limit cycles, based on the fact that for this value of
the parameter it has the hyperbola $xy+1=0$ invariant by the flow. From the above
proposition  it follows that $b^*<1$. This result already improves the upper bound
of $b^*$, given in~\cite{Xian}, $\sqrt[6]{9\pi^2/16}\approx 1.33.$
Theorem~\ref{mteo} improves again this upper bound, but as we will see, the proof
is much more involved.
\end{remark}

\begin{proof}[Proof of Theorem \ref{mteo}] Recall that for $a\le0$ the
function $V(x,y)=x^2+y^2$ is a global
Lyapunov function for system~\eqref{sisa} and therefore the origin is global
asymptotically stable. Then it is easy to see that
its phase portrait is like (o) in Figure~\ref{Sphere}.

To prove the theorem we list some of the key points that we will use and that will be
proved in the forthcoming sections:

\begin{itemize}
\item[(${\mathbf R_1}$)] System~\eqref{sisb} has at most one  limit cycle
for $b\in(0,0.817]$ and when it exists it is hyperbolic and attractor, see
Section~\ref{uni}.

\item[(${\mathbf R_2}$)]  System~\eqref{sisb} has an odd number of limit
cycles, with multiplicities taken into account, when $b\le0.79$ and
 the configuration of its
separatrices is like (i) in Figure~\ref{posRela}, see Proposition~\ref{ppp}.

\item[(${\mathbf R_3}$)] System~\eqref{sisb} has an even number of limit
cycles, with multiplicities taken into account, when $b=0.817$ and the
configuration of its separatrices is like (v) in Figure~\ref{posRela}, see again
Proposition~\ref{ppp}.
\end{itemize}

The theorem for $b\ge b^*$ is a consequence of Proposition~\ref{fr}.
Notice that again by this proposition and (${\mathbf R_3}$),
$b^*<0.817$. Hence, the limit cycles can exist only when
$b\in(0,b^*)\subset(0,0.817]$ and by (${\mathbf R_1}$) when they
exist then there is only one and it is hyperbolic and attractor.

As a consequence of (${\mathbf R_2}$) and the uniqueness and hyperbolicity of the
limit cycle  we have that the phase portrait for $b\le0.79$ is like (i) in
Figure~\ref{Sphere}.

To study the phase portraits for the remaining values of $b$, that
is $b\in(0.79,b^*)$, first notice that all of them have exactly one
limit cycle, which is hyperbolic and stable. So it only remains to
know the behavior of the infinite separatrices. We denote by
$x_2(b)$ and ${x}'_3(b)$ the points of intersection of the
separatrices $\mathcal{S}_2$ and $\mathcal{S}'_3$ of system
\eqref{sisb} with the $x$-axis (when they exist), see also the
forthcoming Figure~\ref{BarS3Com}. Notice that for $b>b^*$,
$x_3'(b)<x_2(b)<0$ and $x_3'(b^*)=x_2(b^*)<0.$ The properties of the
SCFRVF imply that $x_2(b)$ is monotonous increasing  and that
$x_3'(b)$ is monotonous decreasing. Hence for $b\lesssim b^*$ the
phase portrait of the system is like~(iii) in Figure~\ref{Sphere}.
Since we already know that for $b=0.79$ the phase portrait is like
(i), it should exists at least one value, say $b=\hat b$, with phase
portrait (ii). Since for SCFRVF the solution for a given value of
$b$, say $b=\bar b$, becomes a  curve without contact for the system
when $b\ne\bar b$, we have that the phase portraits corresponding to
heteroclinic orbits, that is (ii) and (iv) of Figure~\ref{Sphere},
only appear for a single value of $b$ (in this case $\hat b$ and
$b^*$, respectively). Therefore, the theorem follows.
\end{proof}

\section{Uniqueness of the limit cycle for $b\le817/1000$}\label{uni}
In this section we will prove the uniqueness of the limit cycle of
system \eqref{sisb} when $b\le0.817$. The idea of the proof is to
find a suitable rational Dulac function for applying the following
generalization of Bendixson--Dulac criterion.

\begin{proposition}\label{GenBenDul}
Consider the $C^1$-differential system
\begin{equation}\label{sisPQ}
\left\{\begin{array}{l}
\dot{x}=P(x,y),\\
\dot{y}=Q(x,y),
\end{array}\right.
\end{equation}
and let $\mathcal{U}\subset\R^2$ be   an open region  with boundary formed by
finitely many algebraic curves. Assume that:
\begin{enumerate}[(I)]
\item There exists a  rational function $V(x,y)$
such that
\begin{equation}\label{M} M:=\frac{\partial V}{\partial
x}P+\frac{\partial V}{\partial y}Q-V\left(\frac{\partial P}{\partial
x}+\frac{\partial Q}{\partial y}\right)
\end{equation}
does not change sign on $\mathcal{U}$. Moreover $M$ only vanishes on  points, or
curves that are not invariant by the   flow of~\eqref{sisPQ}.

\item All the connected components of
 $\mathcal{U}\setminus\{V=0\}$, except perhaps one, say \,$\widetilde{\mathcal U}$,  are
 simple connected.
 The  component \,$\widetilde{\mathcal U}$, if exists,  is 1-connected.
\end{enumerate}
 Then the system  has at
most one limit cycle in $\mathcal{U}$ and when it  exists is  hyperbolic and it is
contained in\, $\widetilde{\mathcal U}$. Moreover its stability is given by the
sign of $-VM$ on~\,$\widetilde{\mathcal U}$.
\end{proposition}
The above statement is a simplified version of the one given in \cite{Ga-Gi-2}
adapted to our interests. Similar results  can be seen in \cite{che,Ga-Gi,Llo,
Ya}.

\begin{remark}\label{rrr} Looking at the proof of
Proposition~\ref{GenBenDul} we also know that:
\begin{enumerate}[(i)]
\item The Dulac function used in the proof is $1/V.$
\item In the region $\mathcal{U}$, the curve $\{V(x,y)=0\}$ is without contact for the
flow of~\eqref{sisPQ}. In particular, by the Bendixson-Poincar\'{e}
Theorem,  the ovals of the set $\{V(x,y)=0\}$ must surround some of
the critical points of the vector field.
\end{enumerate}
\end{remark}

To give an idea of how we have found the  function $V$  that we will use in our
proof we will first study  the van der Pol system and then the uniqueness in our
system when $b\le 0.615$. Although we will not use these  two results, we believe
that to start studying them helps to a better understanding of our approach.

\subsection{The van der Pol system}

 Consider the Van der Pol system
\begin{equation}\label{vanderpol}
\left\{\begin{array}{l}
\dot{x}=y,\\
\dot{y}=-x+(b^2-x^2)y.
\end{array}\right.
\end{equation}

Due to the expression of the above family of differential equations,
in order to apply Proposition~\ref{GenBenDul}, it is natural to
start considering functions of the form
\[
V(x,y)=f_2y^2+f_1(x)y+f_0(x).
\]
For this type of functions, the corresponding $M$ is a polynomial of degree 2 in
$y$, with coefficients being functions of $x$. In particular the coefficient of
$y^2$ is
\[
f_1'(x)+f_2(b^2-x^2).
\]
Taking  $f_1(x)=(x^2-3b^2)f_2x/3$ we get that it vanishes. Next,
fixing $f_2=6$, and imposing to the coefficient of $y$ to be  zero
we obtain that $f_0(x)=6x^2+c,$ for any constant $c.$ Finally,
taking $c=b^2(3b^2-4),$ we arrive to

\begin{equation}\label{vbvander}
V_b(x,y)=6y^2+2(x^2-3b^2)xy+6x^2+b^2(3b^2-4).
\end{equation}
From \eqref{M} of Proposition~\ref{GenBenDul}, the corresponding $M$, which only
depends on $x$, is
$$M_b(x,y)=4x^4+b^2(3b^2-4)(x^2-b^2).$$

It is easy to see that  for $b\in(0,2/\sqrt3)\approx(0,1.15)$,
$M_b(x,y)>0$. Notice that $V_b(x,y)=0$ is quadratic  in $y$ and so
is not difficult to see that it has at most one oval, see
Figure~\ref{GrafVander} for $b=1.$ Then we can apply
Proposition~\ref{GenBenDul} to prove the uniqueness and
hyperbolicity of the limit cycle for these values of $b$.

\begin{figure}[h]
\centering\epsfig{file=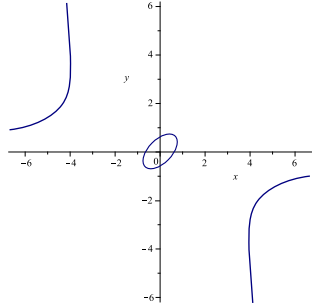,width=4cm,height=4cm}
\caption{The algebraic curve $V_b(x,y)=0$ with $b=1$.}
\label{GrafVander}
\end{figure}

We remark that taking a more suitable polynomial Dulac function, it
is possible to prove the uniqueness of the limit cycle for all
values of $b$, see \cite[p. 105]{chi}. We have only included this
explanation as a first step towards the construction of a suitable
rational Dulac function for our system~\eqref{sisb}.

\subsection{System~\eqref{sisb} with $b\le 651/1000$}\label{651}

 By making some
modifications to the function $V_b$ given by \eqref{vbvander}, we
get an appropriate function for system \eqref{sisb}. Consider

\begin{align*}
V_b(x,y)=&[2x^3+6b^2(1-b^2)x]  y^3+6(1-b^2)y^2+2(x^2-3b^2)xy\\&+
6(1-b^2)x^2+b^2(3b^2-4).
\end{align*}
Computing the double discriminant $\triangle^2(V_b)$ of the function
$V_b$, introduced in Appendix II, we get that
 \[\triangle^2(V_b)=b^2(3b^2-4)(b^2-1)^{15}(P_{19}(b^2))^2,\]
where $P_{19}$ is a polynomial of degree
 19. By using for instance the Sturm method,
 we prove that the  smallest  positive root of $\triangle^2(V_b)$ is greater than $0.85$.
 Therefore by Proposition~\ref{ddisc} we know that for
 $b\in(0,0.85]$ the algebraic curve $V_b(x,y)=0$ has no singular points and
therefore the set $\{V_b(x,y)=0\}\subset\R^2$ is a finite disjoint union of ovals
and smooth curves diffeomorphic to open intervals.

 By applying Proposition~\ref{GenBenDul} to system \eqref{sisb} with $V=V_b$,
we get that
\begin{equation}\label{Mb651}
\begin{array}{lll}
M_b(x,y)&=&6[(2-3b^2)x^4y^2-2b^2(2-b^2)x^3y^3+(2-b^2)x^2y^4]
+2(2-3b^2)x^4\\
&&-3b^2(14-15b^2)x^2y^2+12b^4(2-b^2)xy^3-b^2(4-9b^2)x^2\\
&&+3b^4(2-3b^2)y^2+b^4(4-3b^2).
\end{array}
\end{equation}

In Subsection~\ref{cas2} of Appendix II we prove that $M_b$ does not
vanish on $\R^2$ for $b\in(0,0.651]$. Then by Remark~\ref{rrr} all
the ovals of $\{V_b(x,y)=0\}$ must surround the origin, which is the
unique critical point of the system. Since the straight line $x=0$
has at most two points on the algebraic curve $V_b(x,y)=0$, it can
have at most one closed oval surrounding the origin. Then by
 Proposition~\ref{GenBenDul} it follows the uniqueness, stability and
hyperbolicity of the limit cycle of system~\eqref{sisb} for these values of the
parameter~$b$.

\subsection{System~\eqref{sisb} with $b\le 817/1000$}

The hyperbola $xy+1=0$ will play an important role in the  study of this case. We
first  prove a preliminary result.

\begin{lemma}\label{b=1} Consider  system~\eqref{sisb}.
\begin{enumerate}[(I)]
\item For $b\ne1$ the hyperbola $xy+1=0$ is without contact for its flow.
In particular its periodic orbits never cut it.
\item For $b=1$ the hyperbola  $xy+1=0$ is invariant for its flow
and the system  has not periodic orbits.
\end{enumerate}

\end{lemma}
\begin{proof}
Define $F(x,y)=xy+1$ and set $X=(P,Q):=(y,-x+(b^2-x^2)(y+y^3))$. Simple
computations give that for $x\ne0$,
\[
\left.(F_xP+F_yQ)\right|_{y=-1/x}=\frac{1+x^2}{x^2}\left(1-b^2\right). \]
Therefore (I)  follows and we have also proved that when $b=1$, the hyperbola is
invariant by the flow.

(II) When $b=1$,
\begin{equation}\label{invc1}
F_xP+F_yQ=KF,
\end{equation}
where $K=K(x,y)=y^2-x^2-xy(xy-1)$ is the so called {\it cofactor} of
the invariant curve $F=0$.

Let us prove that the system has no limit cycle. Recall that the origin is
repeller. Therefore if we prove that any periodic orbit $\Gamma$ of the system is
also repeller we will have proved that there is no limit cycle.

This will follow if we show that
\begin{equation}\label{dd}
\int_0^{T} \operatorname{div}(X)(\gamma(t))\,dt>0,
\end{equation}
where  $\gamma(t):=(x(t),y(t))$ is the time parametrization of $\Gamma$ and
$T=T(\Gamma)$ its period.

To prove~\eqref{dd} notice that the divergence of  $X$ can be written as
$\operatorname{div}(X)=3K+2x^2+1-3xy$. Then,
$$
\int_0^{T} \operatorname{div}(X)(\gamma(t))\,dt= 3\int_0^{T}K(x(t),y(t))dt+
\int_0^{T} (2x(t)^2+1)dt-3 \int_0^{T}x(t)y(t)dt.
$$
Observe that from \eqref{invc1} we have that
$$
\begin{array}{lll}
\displaystyle \int_0^{T}K(x(t),y(t))\,dt
&\displaystyle=&\displaystyle\int_0^{T}\frac{F_x(x(t),y(t))\dot
x+F_y(x(t),y(t))\dot y}{F(x(t),y(t))}dt\\
&\displaystyle=&\displaystyle\int_0^{T}\frac{d}{dt}\ln |F(x(t),y(t))|
dt=\ln|F(x(t),y(t))|\Big |_0^{T}=0
\end{array}
$$
and that
$$
\int_0^{T}x(t)y(t)dt =\int_0^{T}x(t)\dot{x}(t)dt=\frac{x^2(t)}{2}\Big|_0^{T}=0.
$$
Therefore
$$
\int_0^{T} \operatorname{div}(X)(\gamma(t))\,dt= \int_0^{T} (2x(t)^2+1)dt>0,
$$
as we wanted to see.
\end{proof}

\begin{theorem} System~\eqref{sisb} for $b\in(0,0.817]$ has at most one limit
cycle. Moreover when it exists it is hyperbolic and attractor.
\end{theorem}

\begin{proof} Based on the function $V_b$ used in the Subsection~\ref{651}
 we consider the function $V_b(x,y)=\widehat V_b(x,y)/( 5+6b^{18}x^2)$, where
\begin{equation}
\begin{array}{lll}
\widehat
V_b(x,y)&=&\frac{1}{2}\,{b}^{18}{x}^{6}+\frac{1}{2}\,{b}^{18}{x}^{4}{y}^{2}+
\left( 1+\frac{1}{2}\,{b}^{ 12} \right) {x}^{3}{y}^{3}+ \left(
1+\frac{3}{2}\,{b}^{2} \right) {x}^{3}y\\&&-
 \left(\frac{3}{5}\,{b}^{10}+\frac{5}{3}\,{b}^{14}+2\,{b}^{16} \right){x}^{2}{y}^{2
}+ \left( 3\,{b}^{2}-3\,{b}^{4}+{\frac {21}{10}}\,{b}^{6} \right) x{y} ^{3}\\&&+
\left( 3-3\,{b}^{2}+2\,{b}^{4} \right) {x}^{2}-{b}^{
2}\left(3-\frac{1}{10}\,{b}^{4} \right) xy+ \left( 3-3\,{b}^{2}+2\,{b}^{4} \right)
{y }^{2}\\&&+\frac{3}{2}\,{b}^{4}-2\,{b}^{2}.
\end{array}
\end{equation}
We have added the non-vanishing denominator to increase a little bit
the range of values for which Proposition~\ref{GenBenDul} works.
Indeed, it can be seen that the above function, but without the
denominator, is good for showing that the system has at most one
limit cycle for $b\le 0.811.$

To study the algebraic curve $\widehat V_b(x,y)=0$ we proceed like in the previous
subsection. The double discriminant introduced in Appendix II is
 \[\triangle^2(\widehat V_b)=b^{182}(3b^2-4)(4b^{36}+27b^{24}+108 b^{12}+108)
 (P_{152}(b^2))^2,\]
where $P_{152}$ is a polynomial of degree
 152. It can be seen that
 the  smallest  positive root of $\triangle^2(\widehat V_b)$ is greater than $0.88$.
 Therefore by Proposition~\ref{ddisc} we know that for
$b\in(0,0.88]$ this algebraic curve  has no singular points. Hence the set
$\{V_b(x,y)=0\}\subset\R^2$ is a finite disjoint union of ovals and smooth curves
diffeomorphic to open intervals.

The function  that we have to study in order to apply Proposition~\ref{GenBenDul}
is
\begin{equation}\label{ap1}
M_b(x,y)=\frac{N_b(x,y)}{30(6b^{18}x^2+5)^2}
\end{equation}
where $N_b(x,y)$ is given in \eqref{Mb817} of Subsection~\ref{cas3}.
The denominator of $M_b$ is positive for all $(x,y)\in\R^2$. By
Lemma~\ref{b=1} we know that the limit cycles of the system must lay
in the open region $\Omega=\R^2\cap\{xy+1> 0\}$. In
Subsection~\ref{cas3} of Appendix II we will prove that $N_b$ does
not change  sign on the region $\Omega$ and if it vanishes it is
only at some isolated points.

Notice also that the set $\{\widehat V_b(x,y)=0\}$ cuts the $y$-axis at most in
two points, therefore by the previous results and  arguing as in
Subsection~\ref{651}, we know that it has at most one oval and that when it exists
it must surround the origin.

Therefore we are under the hypotheses of Proposition~\ref{GenBenDul}, taking
$\mathcal{U}=\Omega,$  and the uniqueness and hyperbolicity of the limit cycle
follows.
\end{proof}

\section{Phase portraits for $b\le79/100$ and $b=817/1000$}\label{pp}

This section is devoted to find the relative position of the
separatrices of the infinite critical points when $b\le0.79$ and
when $b=0.817$. The main tool will be the construction of algebraic
curves that are without contact for the flow of system \eqref{sisb}.
These curves are essentially obtained by using  the functions
$\phi_i(x):=\tilde{\phi}_i(x-b)/(x-b)^2$ and
$\varphi_i(x):=\tilde{\varphi}_i(1/x)$ where $\tilde{\phi}_i$ and
$\tilde{\varphi}_i$ are the approximations of order $i$ of the
separatrices of the infinite critical points, given in the
expressions~\eqref{sepS2S3} and~\eqref{sepS1} of Theorem \ref{infi},
respectively. That is, we
 use algebraic approximations of $\mathcal{S}_i$ and  $\mathcal{S}'_i,$ for
 $i=1,2,3.$

As usual for knowing when a vector field $X$ is without contact with
a curve of the form $y=\psi(x)$ we have to control the sign of
$$
N_{\psi}(x):=\langle\nabla(y-\psi(x)),X\rangle\big{|} _{y=\psi(x)}.
$$
In this section we will repeatedly  compute this function when $\psi(x)$ is either
$\varphi_i(x)$, $\phi_i(x)$ or modifications of these functions.

We prove the following result.

\begin{proposition}\label{ppp} Consider system \eqref{sisb}. Then:
\begin{enumerate}[(I)]
\item For $b\le{79}/{100}$ the configuration of its
separatrices is like (i) in Figure~\ref{posRela}. Moreover it has an odd number of
limit cycles, taking into account their multiplicities.

\item For $b={817}/{1000}$ the configuration of its
separatrices is like (v) in Figure~\ref{posRela}. Moreover it has an even number
of limit cycles, taking into account their multiplicities.
\end{enumerate}
\end{proposition}

\begin{proof} (I)  Consider the two functions
$$\varphi_1(x) = -\frac{1}{x}, \qquad \textrm{and}
\qquad \varphi_2(x) = -\frac{1}{x}-\frac{(b^2-1)}{x^3},
$$
which are the corresponding expressions in the plane $(x,y)$ of the
first and second approximation of the separatrix $\mathcal{S}_1$.

If $b<1$ then $(\varphi_1-\varphi_2)(x)=(b^2-1)/x^3>0$ for $x<0$.
This implies that the separatrix $\mathcal{S}_1$ in the
$(x,y)$-plane and close to $-\infty$ is below the graphic of
$\varphi_1(x)$. Moreover
$$
N_{\varphi_1}(x)=-\frac{(x^2+1)(b^2-1)}{x^3}<0 \qquad \mbox{for
$x<0$}.
$$
This inequality implies that the separatrix $\mathcal{S}_1$ in the plane $(x,y)$
cannot intersect the graphic of $\varphi_1(x)$ for $x<0$, see Figure~\ref{retrs1}.
\begin{figure}[h]
\begin{center}
\epsfig{file=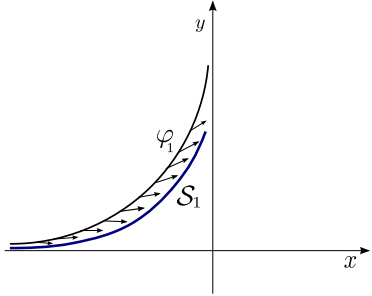,width=4.8cm} \caption{Behavior of
$\mathcal{S}_1$ for $b<1$.} \label{retrs1}
\end{center}
\end{figure}

Now, we consider the third approximation to the separatrices $\mathcal{S}_2$ and
$\mathcal{S}_3$, that is we consider the first three terms in \eqref{sepS2S3}. It
is  given by the graph of the function
$$
\phi_3(x)=\frac{(x^2-5bx+13b^2)}{9b^3(x-b)^2}.
$$

Let us prove that when $b\in(0,\sqrt{2/3})$,  the graphs of $\varphi_1(x)$ and
$\phi_3(x)$ intersect at a unique point, $(x_0,y_0)$ with $x_0<0$ and $y_0>0$. For
this is sufficient to show that the function $(\varphi_1-\phi_3)(x)$ has a unique
zero at some $x_0<0$.

It is clear that $\lim_{x\rightarrow 0^{-}}(\varphi_1-\phi_3)(x)=+\infty$ and we
have that $(\varphi_1-\phi_3)(-2b)=(3b^2-2)/6b^3$, then for $b<\sqrt{2/3}$,
$(\varphi_1-\phi_3)(-2b)<0$ hence $(\varphi_1-\phi_3)(x)$ has a zero at a point
$x_0$ with $-2b<x_0<0$. Moreover this zero is unique because the numerator of
$(\varphi_1-\phi_3)(x)$ is a monotonous function.

It also holds that   $\nabla (y-\phi_3(x))=(-\phi_3'(x),1)$ where
$\phi_3'(x)=(-x+7b)/3b^2(-x+b)^3$ is a positive function for $x<0$, and a simply
computation shows that
\begin{eqnarray*}
\begin{array}{lll}
N_{\phi_3}(x)&=&
\frac{1}{-729b^9(b-x)^2}\left[(81b^6+1)x^4+(729b^8-405b^6-11)bx^3\right.
\\&&-9(162b^8-108b^6-7)b^2x^2+(729b^8+405b^6-178)b^3x
\\&&\left.-13(81b^6-20)b^4\right].
\end{array}
\end{eqnarray*}
To control the sign of $N_{\phi_3}$ we compute the discriminant of its numerator
with respect to $x$. It gives $\operatorname{dis}(N_{\phi_3}(x),x)=
b^{12}P_{22}(b^2)$, where $P_{22}$ is a polynomial of degree 22 with integer
coefficients.

By  using the Sturm  method  we obtain that $P_{22}(b^2)$  has exactly four real
zeros. By Bolzano theorem the positive ones   belong to the intervals
$(0.7904,0.7905)$ and $(2.6,2.7)$.

\begin{figure}[h]
\begin{center}

\begin{tabular}{ccc}
\epsfig{file=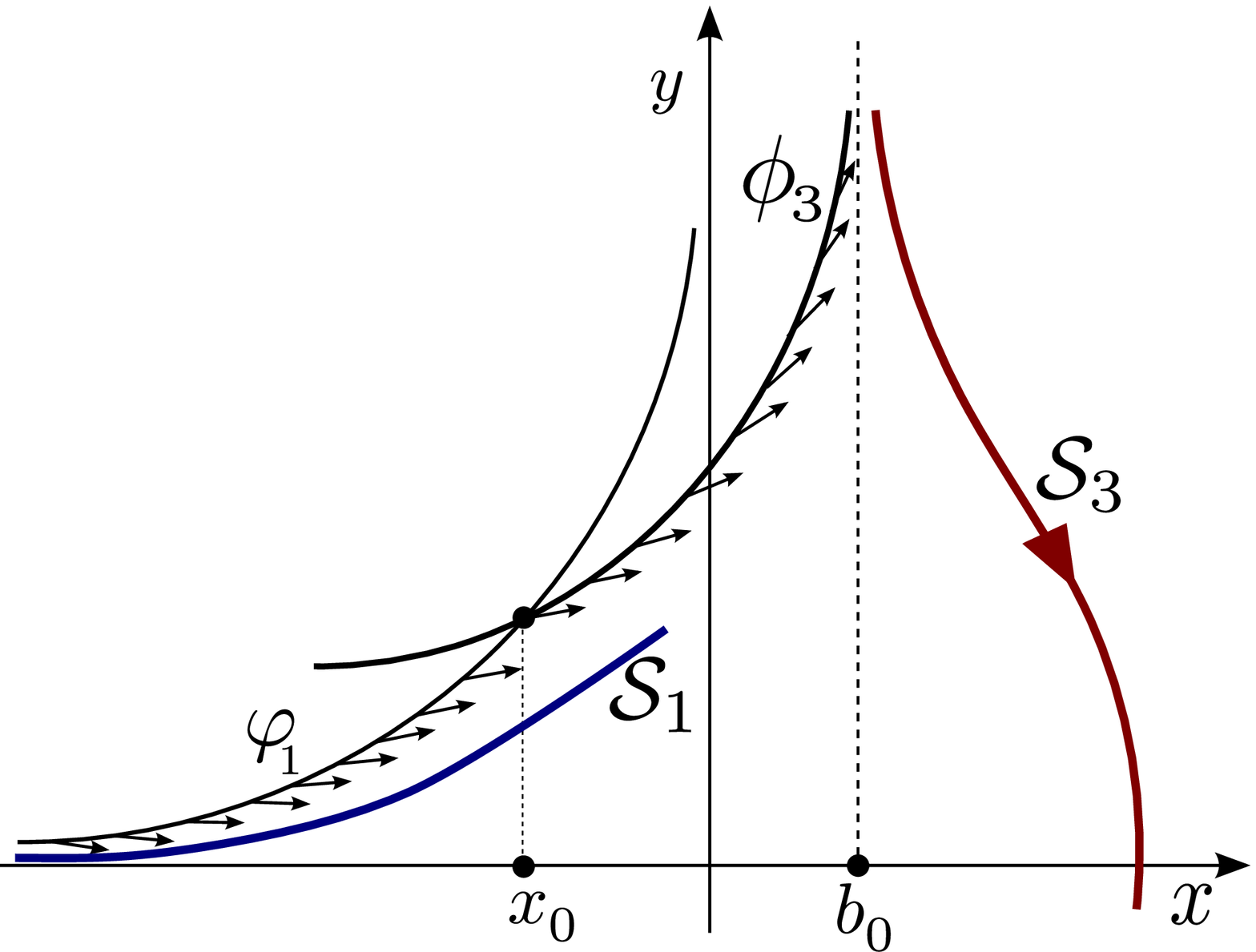,width=5cm}& &
\epsfig{file=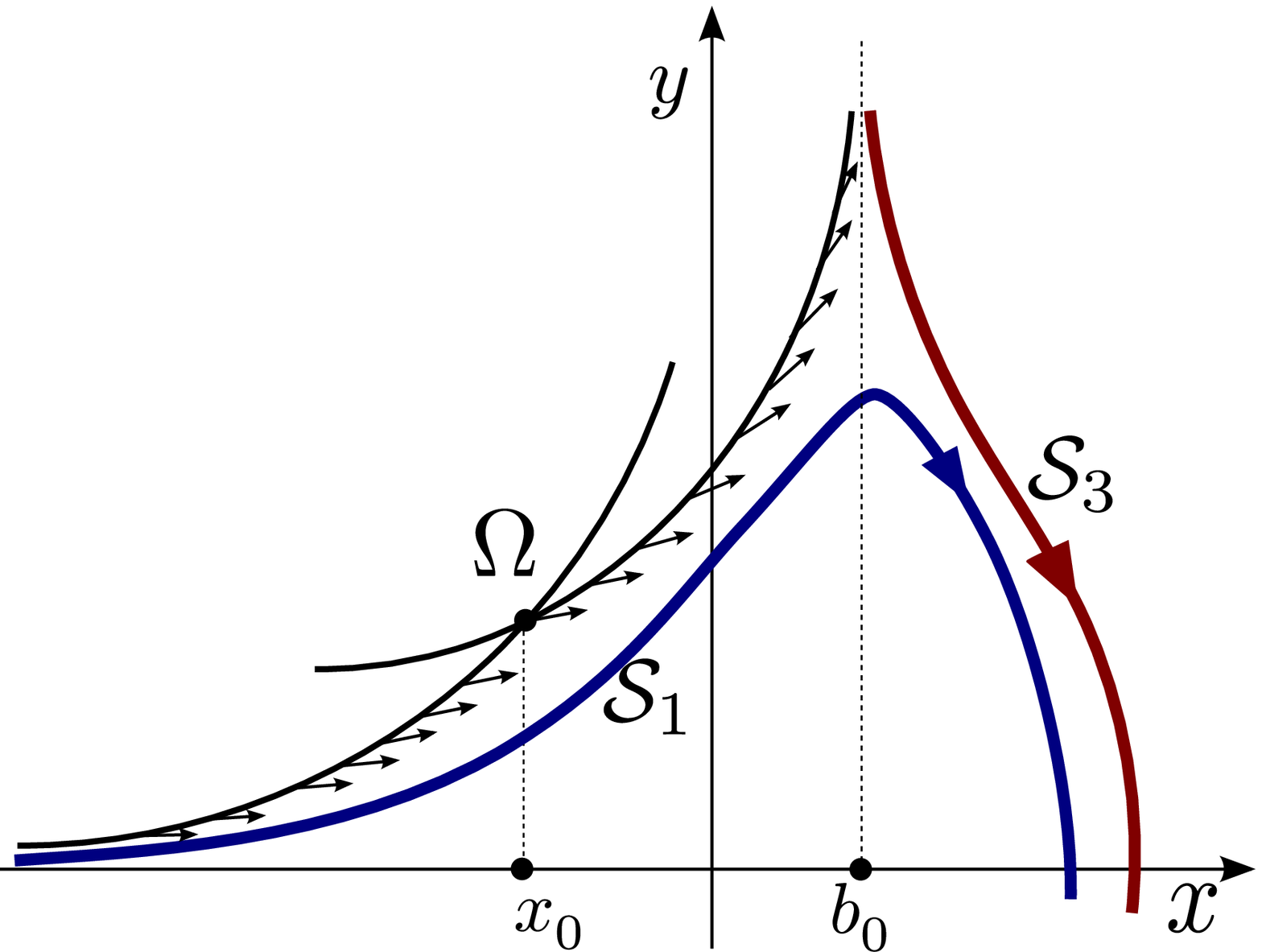,width=5cm}\\
(a) & & (b)
\end{tabular}
\caption{Behavior of $\mathcal{S}_1$ and $\mathcal{S}_3$ for
$b\le0.79$ } \label{retrs11}
\end{center}
\end{figure}

If we fix  $b_0\le79/100$ then  $b_0<\sqrt{2/3}$ and moreover according to
previous paragraph the graphics of $\varphi_1(x)$ and $\phi_3(x)$ intersect at a
unique point $(x_0,y_0)$ with $x_0<0$ and $y_0>0$. Furthermore, $\frac{\partial
N_{\phi_3}}{\partial b}(b_0)>0$ in $(x_0,b_0)$ and $N_{\phi_3}<0$ in $(x_0,b)$ for
all $b \in (0,b_0]$. Therefore the vector field associated to \eqref{sisb} on
these curves is the one showed in Figure~\ref{retrs11}.(a).

From Figure~\ref{retrs11}.(a) it is clear that the separatrix $\mathcal{S}_1$
cannot intersect the set $\Omega=\{(x,\varphi_1(x))|-\infty<x\leq x_0\}\cup
\{(x,\phi_3(x))|x_0\leq x<b_0\}$. Moreover, since the separatrix $\mathcal{S}_2$
forms an hyperbolic sector together with~$\mathcal{S}_3$ we obtain that
$\mathcal{S}_1$ cannot be asymptotic to the line $x=b_0$. Hence we must have the
situation showed in Figure~\ref{retrs11}.(b). We know that the origin is a source
and from the symmetry of system \eqref{sisb} we conclude that for $b\leq 0.79$ the
system has an odd number of limit cycles (taking into account multiplicities) and
the phase portrait  is the one showed  in Figure~\ref{retrsf}.
\begin{figure}[h]
\begin{center}
\epsfig{file=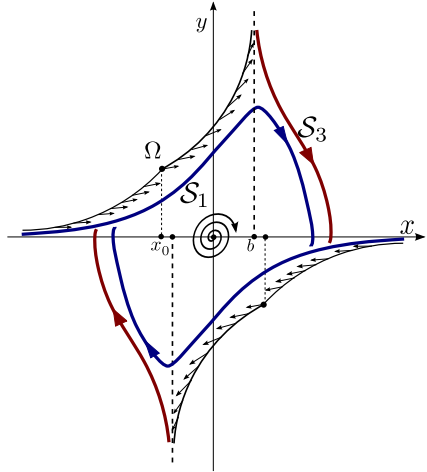,width=5.5cm} \caption{For $0<b\leq 0.79$,
system \eqref{sisb} has at least one  limit cycle and phase
portrait~(i) of Figure~\ref{Sphere} .} \label{retrsf}
\end{center}
\end{figure}

\smallskip

(II) We start proving the result  when $b=b_0:=89/100$ because the
method that we use is the same that for studying the case
$b=817/1000$, but the computations are easier. Recall that we want
to prove that the configuration of separatrices is like  (v) in
Figure~\ref{posRela}. That the number of limit cycles must be even
(taking into account multiplicities) is then a simply consequence of
the Poincar\'{e}--Bendixson Theorem, because the origin is a source.

We  consider the approximation of eight order to $\mathcal{S}_2$ and
$\mathcal{S}_3$ given by the graph of the function $\phi_{8}(x)$.

By using again the Sturm method it is easy to see that $N_{\phi_{8}}(x)<0$ for
$x\in(b_0,x_0)$, where $x_0=1.924$ is a left approximation to the root of the
function $\phi_{8}(x)$, and $N_{\phi_{8}}(x)>0$ for $x\in(x_1,b_0)$, where
$x_1=-2.022$ is  a right approximation to the root of the function
$N_{\phi_{8}}(x)$. That is, we have the situation shown in Figure~\ref{bar1}.(a).
Now, we consider the function $\hat{\phi}_{8}(x)=\phi_{8}(x)-1/(9b^3)$, is clear
that $(\phi_{8}-\hat{\phi}_{8})(x)>0$. We have $N_{\hat{\phi}_{8}}(x)>0$ for
$x\in(b_0,x_2)$ where $x_2=1.6467$ is a left approximation to the root of the
function $\hat{\phi}_{8}(x)$, moreover the line $x=x_2$ is transversal to the
vector field for $y>0$, thus the separatrix $\mathcal{S}_3$ intersects the
$x$-axis at a point $\bar{x}$ of the interval $(x_2,x_0)$, see again
Figure~\ref{bar1}.(a).
\begin{figure}[h]
\begin{center}
\begin{tabular}{cc}
\epsfig{file=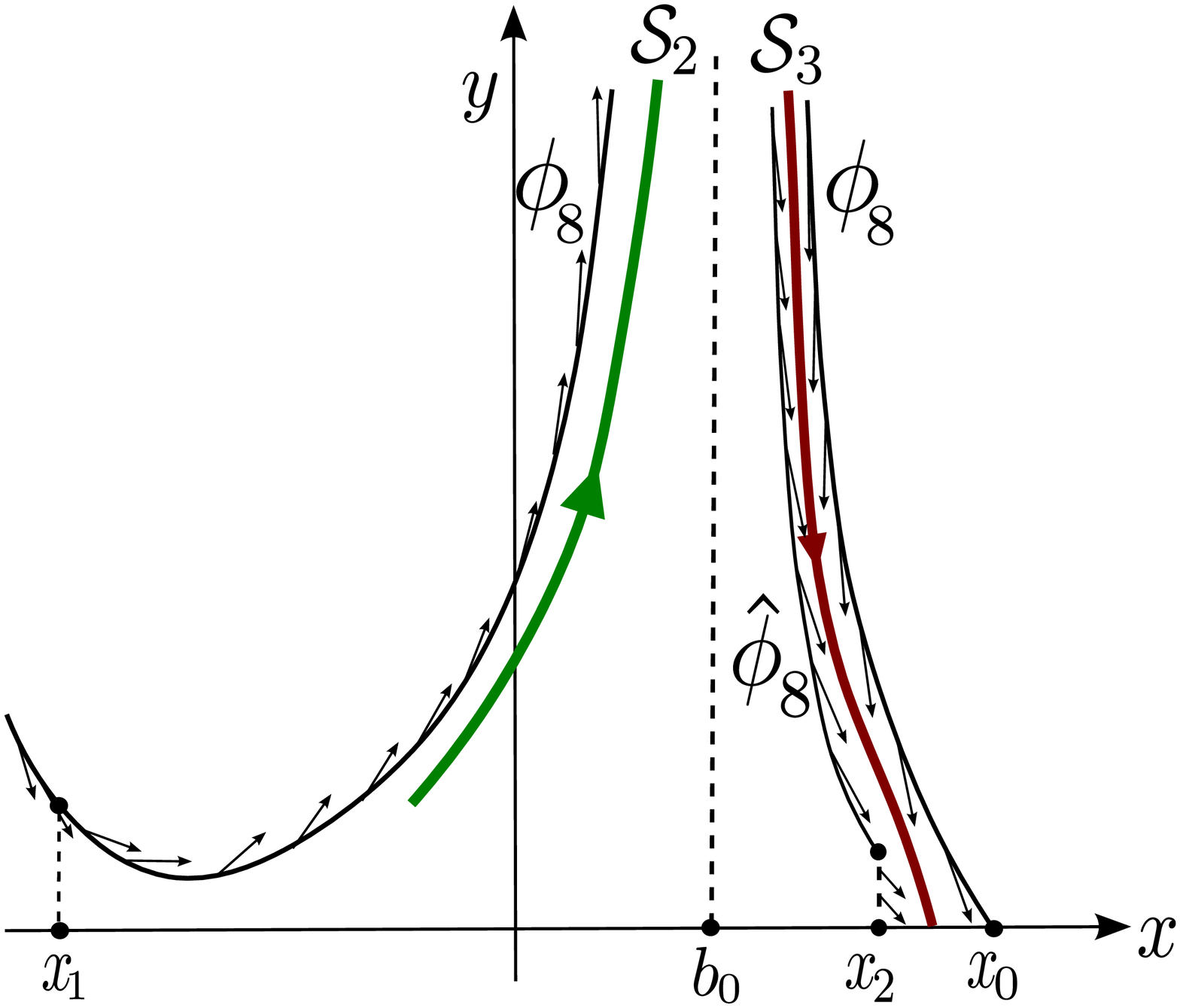,width=6cm}&
\epsfig{file=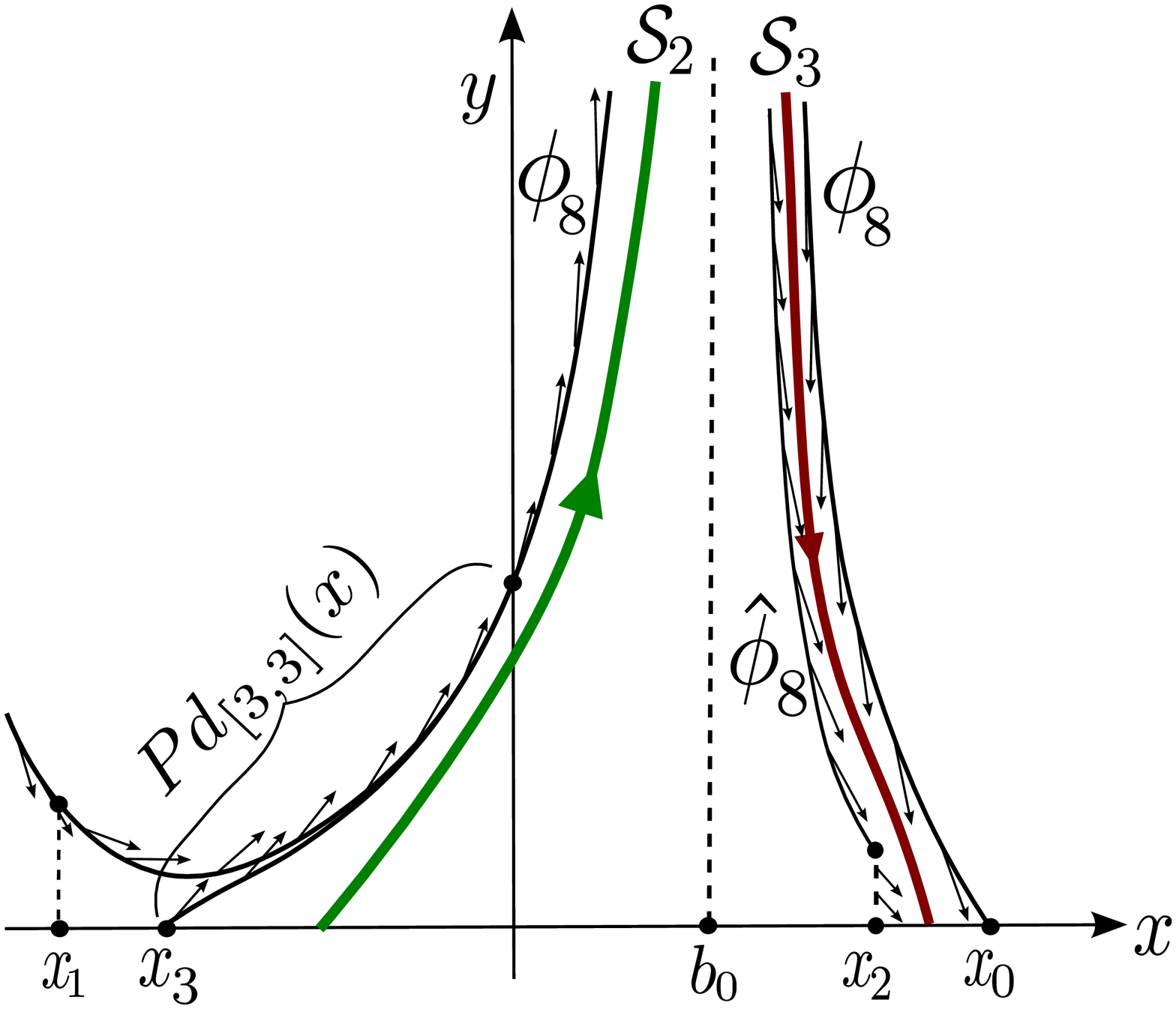,width=6cm}\\
(a) &  (b)
\end{tabular}
\end{center}\caption{Behaviour of $\mathcal{S}_2$ and
$\mathcal{S}_3$ for $b\in\{0.817,0.89\}$.} \label{bar1}
\end{figure}

At this point, the idea is to show that $\mathcal{S}_2$ intersects
the $x$-axis at a point $\hat{x}$, with $-x_2<\hat{x}<0$. For
proving this, we  utilize the Pad\'e approximants method,
see~\cite{pade}.

 Recall that given a function $f(x)$,
its Pad\'e approximant $\operatorname{Pd}_{[n,m]}(f)(x,x_0)$ of order $(n,m)$ at a
point $x_0$, or simply $\operatorname{Pd}_{[n,m]}(f)(x)$ when $x_0=0$, is a
rational function of the form $F_n(x)/G_m(x)$, where $F_n$ and $G_m$ are
polynomials of degrees $n$ and $m$, respectively, and such that
$$
\left|f(x)-\frac{F_n(x)}{G_m(x)}\right|=O\left((x-x_0)^{n+m+1}\right).
$$
Consider the Pad\'e approximant $\operatorname{Pd}_{[3,3]}(\phi_8).$
It satisfies that $\operatorname{Pd}_{[3,3]}(\phi_8)(0)=\phi_8(0)$
 and by the Sturm method it can be
seen that there exists $x_3<0$ such that
$\operatorname{Pd}_{[3,3]}(\phi_8)(x_3)=0$,
 $\operatorname{Pd}_{[3,3]}(\phi_8)$ is positive and increasing
on the interval $(x_3,0)$ and a left approximation to $x_3$ is $-1.595$. Moreover
it is easy to see that $N_{\operatorname{Pd}_{[3,3]}(\phi_8)}(x)>0$ for
$x\in(x_3,0)$. Therefore $\mathcal{S}_2$ cannot intersect neither the graph of
$y=\operatorname{Pd}_{[3,3]}(\phi_8)(x)$  in $(x_3,0)$ nor the  graph of
$\phi_8(x)$ in $[0,b_0).$ Hence $\mathcal{S}_2$ intersects  the $x$-axis in a
point $\hat{x}$ contained in the interval $(x_3,0)$. This implies that
$-x_2<\hat{x}<0$  as we wanted to see, because $-x_2<x_3.$ Hence the behavior of
the separatrices is like Figure~\ref{bar1}.(b). See also Figure~\ref{BarS3Com}.

\begin{figure}[h]
\begin{center}
\epsfig{file=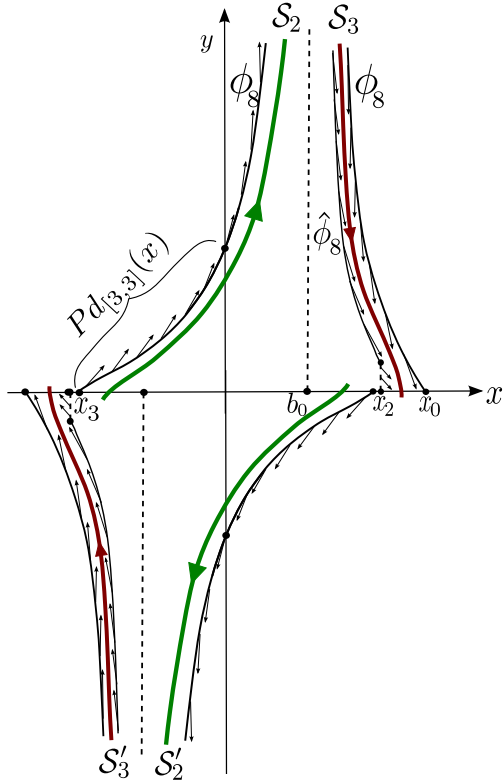,width=6.5cm}
\end{center}\caption{Behavior of $\mathcal{S}_2,\mathcal{S}_3,\mathcal{S}'_2$ and
$\mathcal{S}'_3$ for $b\in\{0.817,0.89\}$.} \label{BarS3Com}
\end{figure}

When $b_0=817/1000$ we follow the same ideas. For  this case we
consider the functions $\phi_{16}(x)$ and
$\hat{\phi}_{16}(x)=\phi_{16}(x)-1/(9b^3)$. Recall that  the graphic
of $\phi_{16}(x)$ is the sixteenth order approximation to
$\mathcal{S}_2$ and $\mathcal{S}_3$. It is  not difficult to prove
that $N_{\hat{\phi}_{16}}>0$  on the interval $(b_0,x_2)$, with
$x_2=1.6421$ and since the line $x=x_2$ is transversal to $X$ for
$y>0$, $\mathcal{S}_3$ intersects the $x$-axis at a point
$\bar{x}>x_2$. Also we have that $N_{\phi_{16}}>0$ on the interval
$(-3/100,b_0)$ and using the Pad\'e approximant
$\operatorname{Pd}_{[5,1]}(\phi_{16})(x,-3/100)$ we obtain that
$\mathcal{S}_2$ intersect to the $x$-axis in a point
$\hat{x}\in(x_3,0)$ with $x_3>-1.638$. This implies that
$-x_2<\hat{x}<0$ as in the case $b=0.89$. Hence we have the same
situation that in Figure~\ref{BarS3Com}.
\end{proof}

\begin{remark} As it is shown in the proof of Theorem~\ref{mteo}, the values
$0.79$ and $0.817$, obtained in the previous proposition, provide a lower and an
upper bound for $b^*.$ We have tried to shrink the interval where $b^*$ lies using
higher order approximations of the separatrices, but we have not been able to
diminish its size.
\end{remark}

\section*{Appendix I: The Descartes method }\label{ddd}

Given a real polynomial $P(x)=a_nx^n+\cdots+a_1x+a_0$ and a real interval
$I=(\alpha,\beta)$ such that $P(\alpha)P(\beta)\ne0,$ there are two well-known
methods for knowing the number of real roots of $P$ in $I$: the Descartes rule and
the Sturm method.

Theoretically, when all the $a_i\in\Q$ and $\alpha,\beta\in\Q$, the
Sturm approach solves completely the problem. If all the roots of
$P$ are simple it is possible to associate to it a sequence  of
$n+1$ polynomials,  the so called {\it Sturm sequence}, and knowing
the signs of this sequence evaluated at $\alpha$ and $\beta$ we
obtain the exact number of real roots in the interval. If $P$ has
multiple roots it suffices to start with $P/(\gcd(P,P'))$, see
\cite[Sec. 5.6]{sb}.

Nevertheless when the rational numbers have big numerators and
denominators and $n$ is also big, the computers have not enough
capacity to perform the computations to get the Sturm sequence. On
the other hand the Descartes rule is not so powerful but a careful
use, in the spirit of bisection method, can many times solve the
problem.

To recall the Descartes rule we need to introduce some notation.
Given an ordered list of real numbers
$[b_0,b_1,\ldots,b_{n-1},b_{n}]$ we will say that  it has $C$
changes of sign if the following holds: denote by
$[c_0,c_1,\ldots,c_{m-1},c_{m}],$ $m\le n$ the new list obtained
from the previous one after removing the zeros and without changing
the order of the remaining terms. Consider the $m$ non-zero numbers
$\delta_i:=c_ic_{i+1}$, $i=0,\ldots,m-1$. Then $C$ is the number of
negative $\delta_i.$

\begin{theorem}[Descartes rule]\label{des} Let $C$
be the number of changes of sign of the list of ordered numbers
\[
[a_0,a_1,a_2,\ldots,a_{n-1},a_n].
\]
Then the number of positive zeros of the polynomial $P(x)=a_nx^n+\cdots+a_1x+a_0$,
counted with their multiplicities, is $C-2k,$ for some $k\in\N\cup\{0\}.$
\end{theorem}

\begin{corollary}\label{cdes} With the notations of Theorem~\ref{des} if
$C=0$ then $P(x)$ has not positive roots and if $C=1$ it has exactly one simple
positive root.
\end{corollary}

In order to apply Descartes rule to arbitrary open intervals we introduce the
following definition:

\begin{definition} Given a real polynomial $P(x)$
and a real interval $(\alpha,\beta)$ we construct a new polynomial
\[
N_\alpha^\beta (P)(x):=(x+1)^{\deg P}P\left(\frac{\beta
x+\alpha}{x+1}\right).
\]
We will call $N_\alpha^\beta (P),$ the {\it normalized version of $P$ with respect
to $(\alpha,\beta)$}. Notice that the number of real roots of $P(x)$ in the
interval $(\alpha,\beta)$ is equal to the number of real roots of  $N_\alpha^\beta
(P)(x)$ in $(0,\infty)$.
\end{definition}

The method suggested in \cite{KM} consists in writing
$(\alpha,\beta)=\bigcup_{i=1}^k(\alpha_i,\alpha_{i+1}),$ with
$\alpha=\alpha_1<\alpha_2<\cdots<\alpha_k<\alpha_{k+1}=\beta$ in
such a way that on each $(\alpha_i,\alpha_{i+1})$ it is possible to
apply Corollary \ref{cdes} to the normalized version of the
polynomial. Although there is no  systematic way of searching a
suitable decomposition, we will see that a careful use of these type
of ideas has been good enough to study the number and localization
of the roots for a huge polynomial of degree 965, see
Subsection~\ref{cas3}  in Appendix II.

\section*{Appendix II: A method for controlling the sign\\ of
polynomials in two variables}\label{app}

The main result of this appendix is a new method for controlling the
sign of families of polynomials with two variables. As a starting
point we prove a simple result for one-parameter families of
polynomials in one variable.

Let $G_b(x)$ be a one-parametric  family of polynomials. As usual, we write
$\triangle_x(P)$ to denote the discriminant
 of a polynomial  $P(x)=a_nx^n+\cdots+a_1x+a_0,$  that is,
$$
\triangle_x(P)=(-1)^{\frac{n(n-1)}{2}}\frac{1}{a_n}\operatorname{Res}(P(x),P'(x)),
$$
where $\operatorname{Res}(P,P')$ is the resultant of $P$ and $P'$.

\begin{lemma}\label{ll}
Let \[ G_b(x)=g_n(b)x^n+g_{n-1}(b)x^{n-1}+\cdots+g_1(b)x+g_0(b),\] be a family of
real polynomials  depending also polynomially on a real parameter~$b$ and set
$\Omega=\R$. Suppose that there exists an open interval $I\subset\R$ such that:
\begin{enumerate}[(i)]
\item There is some  $b_0\in I$, such that $G_{b_0}(x)>0$ on $\Omega$.

\item For all $b\in I,$  $\triangle_x(G_b)\neq 0.$

\item  For all $b\in I$, $g_n(b)\ne0.$
\end{enumerate}
Then for all $b\in I$, $G_b(x)>0$ on $\Omega.$

Moreover if $\Omega=\Omega_b=(c(b),\infty)$ for some smooth function $c(b)$,
  the same result holds changing $\Omega$ by this new $\Omega_b$ if we
add the additional hypothesis
\begin{enumerate}[(iv)]
\item  For all $b\in I$, $G_b(c(b))\ne0.$
\end{enumerate}
\end{lemma}

\begin{proof} The key point of the proof is that the roots (real and complex)  of $G_b$
depend continuously of $b,$ because  $g_n(b)\ne0.$ Notice that
hypotheses (iii) and (iv) prevent that moving $b$ some root enters
in~$\Omega$ either from infinity or from the boundary of~$\Omega$,
respectively. On the other hand if moving $b$ some real roots appear
from~$\C$, they do appear trough a double real root that is detected
by the vanishing of $\triangle_x(G_b)$. Since by item (ii),
$\triangle_x(G_b)\neq 0$ no real root appears in this way. Hence,
for all $b\in I,$ the number of real roots of any $G_b$ is the same.
Since by item (i) for $b=b_0$, $G_{b_0}>0$ on $\Omega$, the same
holds for all $b\in I.$
\end{proof}

To state the corresponding result for  families of polynomials with two variables
inspired in the above lemma, see Proposition~\ref{lem-inf}, we need to prove some
results about the iterated discriminants (to replace hypothesis (ii) of the lemma)
and to recall how to study the infinity of planar curves (to replace hypothesis
(iii)).

\subsection{The double discriminant}
Let $F(x,y)$ be a complex polynomial on $\mathbb{C}^2$. We write $F$
as
\begin{equation}\label{F}
F(x,y)=a_ny^n+a_{n-1}y^{n-1}+a_{n-2}y^{n-2}+\ldots+a_1y+a_0,
\end{equation}
where $a_i=a_i(x)\in \mathbb{C}[x]$. Then
$$
\triangle_y(F)=(-1)^{\frac{n(n-1)}{2}}\frac{1}{a_n}\operatorname{Res}(F,{\partial
F}/{\partial y}),
$$
and this resultant can be
 computed as the determinant of
the Sylvester matrix of dimension $(2n-1)\times(2n-1)$,  see
\cite{cox}, {\footnotesize
$$
 S=\left(\begin{array}{cccccccc}
a_n&0&0&0&na_n&0&0&0\\
a_{n-1}&a_n&0&0&(n-1)a_{n-1}&na_n&0&0\\
a_{n-2}&a_{n-1}&\ddots&0&(n-2)a_{n-2}&(n-1)a_{n-1}&\ddots&0\\
\vdots&&\ddots&a_n &\vdots&&\ddots&na_n\\
&\vdots&&a_{n-1}&&\vdots&&(n-1)a_{n-1}\\
a_0&&&&a_1&&& \\
0&a_0&&\vdots&0&a_1&&\vdots\\
0&0&\ddots&&0&0&\ddots&\\
0&0&0&a_0&0&0&0&a_1
\end{array}\right).$$}

 We will write $\triangle^2_{y,x}(F)=
\triangle_x(\triangle_y(F))$. Analogously we can compute
$\triangle^2_{x,y}(F).$ This so called double discriminant plays a
special role in the characterization of singular curves of
$\{F(x,y)=0\}$ and it is also used in applications, see for instance
\cite{Al-S-Se,Laz,Pe-Ro}. In particular we prove the following
result.

\begin{proposition}\label{ddisc}
Let $F(x,y)$ be a complex polynomial on $\mathbb{C}^2$. If
$\{F(x,y)=0\}\subset\mathbb{C}^2$ has a singular point, that is, if there exists a
point $(x_0,y_0)\in\mathbb{C}^2$ such that $F(x_0,y_0)=\partial
F(x_0,y_0)/\partial x=\partial F(x_0,y_0)/\partial y=0$, then
$\triangle^2_{y,x}(F)=\triangle^2_{x,y}(F)=0$.
\end{proposition}

\begin{proof}
  We write $F(x,y)$ in the form \eqref{F}. Without lost
of generality we assume that $(x_0,y_0)=(0,0)$. Then from the assumptions it
follows that $a_0(0)=a_0'(0)=0$ and $a_1(0)=0$, that is, $a_0(x)=x^2\hat{a}_0(x)$
and $a_1(x)=x\hat{a}_1(x)$, with both  $\hat a_i$ also  polynomials.

By using the Sylvester matrix $S$ defined above, we have that
\begin{equation}\label{det}
\det S=(-1)^n a_0\det (S(2n-1\mid n-1))+a_1\det (S(2n-1\mid 2n-1)),
\end{equation}
where $S(i\mid j)$ means the matrix obtained from $S$ by removing the $i$-th row
and the $j$-th column.

Notice that the elements of the last row of $S(2n-1\mid 2n-1)$ are only $0,a_0$
and $a_1$. Therefore, developing the  determinant of this matrix from this row we
get that $\det (S(2n-1\mid 2n-1))=x Q(x),$ for some polynomial $Q(x).$

Hence, by using~\eqref{det}, we get that  $\det S=x^2 P(x)$ with
$P(x)$ another polynomial. This implies that $\triangle_{y}(F)$ has
a double zero at $x=0$ and hence $\triangle^2_{y,x}(F)=0$.

Analogously we can prove that $\triangle_{x}(F)$ has a double zero
at $y=0$ and hence $\triangle^2_{x,y}(F)=0$.
\end{proof}

\begin{corollary}\label{cccc}
 Consider a one-parameter family of polynomials $F_b(x,y)$, depending also polynomially
on $b.$ The values of $b$ such that the algebraic curve $F_b(x,y)=0$
has some singular point in $\C^2$ have to be zeros of the polynomial
\[
\triangle^2(F_b):=\gcd\big(\triangle^2_{x,y}(F_b),\triangle^2_{y,x}(F_b)\big).
\]
\end{corollary}

By simplicity we will also call  the polynomial $\triangle^2(F_b)$, {\it double
discriminant of the family $F_b(x,y)$}. As far as we know the above necessary
condition for detecting  algebraic curves with singular points is new.

\begin{remark} (i) Notice that if in
Corollary~\ref{cccc}, instead of imposing  that for $b\in I$,
$\triangle^2(F_b)\ne0,$ it suffices to check only that either
$\triangle^2_{x,y}(F_b)\ne0$ or $\triangle^2_{y,x}(F_b)\ne0.$

(ii) The converse of the Proposition \ref{ddisc} is not true. For instance  if we
consider the polynomial $F(x,y)=x^3y^3+x+1$ then
$\triangle^2_{y,x}(F)=\triangle^2_{x,y}(F)=0$, however $F_x(x,y)=3x^2y^3+1$ and
$F_y(x,y)=3x^3y^2$ hence $\{F(x,y)=0\}$ does not have singular points.

(iii) Sometimes $\triangle^2_{y,x}(F)\ne \triangle^2_{x,y}(F)$. For instance this
is the case when $F=y^2+x^3+bx^2+bx$ because
\[
\triangle^2_{x,y}(F)=-110592b^9(b-4)(b-3)^6\quad\mbox{and}\quad
\triangle^2_{y,x}(F)=256b^3(b-4).
\]
Notice that $\triangle^2(F)=b^3(b-4).$
\end{remark}

\subsection{Algebraic curves at infinity}
 Let
$$F(x,y)=F^0(x,y)+F^1(x,y)+\cdots+F^n(x,y)$$
be a polynomial on $\R^2$ of degree~$n$. We denote by
$$\widetilde{F}(x,y,z)=z^nF^0(x,y)+z^{n-1}F^1(x,y)+\cdots+F^n(x,y)$$
its homogenization in $\mathbb{RP}^2$.

For studying $\widetilde{F}(x,y,z)$ in $\mathbb{RP}^2$ we can use
its expressions in the three canonical charts of $\mathbb{RP}^2$,
$\{[x:y:1]\}$, $\{[x:1:z]\}$, and $\{[1:y:z]\}$, which can be
identified with the real planes $\{(x,y)\}$,$\{(x,z)\}$, and
$\{(y,z)\}$ respectively. Of course the expression in the chart
$\{[x:y:1]\}$, that is, in the $(x,y)$-plane is precisely $F(x,y)$.

We denote by $\widetilde{F}_1(x,z)$ and $\widetilde{F}_2(y,z)$ the expressions of
the function $\widetilde{F}$ in the planes $\{(x,z)\}$ and $\{(y,z)\}$,
respectively. Therefore $\widetilde{F}_1(x,z)=\widetilde{F}(x,1,z)$ and
$\widetilde{F}_2(y,z)=\widetilde{F}(1,y,z)$.

Let $[x^*:y^*:z^*]\in \mathbb{RP}^2$ be a point of
$\{\widetilde{F}=0\}$. If $z^*\neq 0$, then $[x^*:y^*:z^*]$
corresponds to a point in $\R^2$, otherwise it is said that
$[x^*:y^*:0]$ is a point of $F$ at infinity. Notice that the points
at infinity of $F$ correspond to the points $[x^*:y^*:0]$ where
$(x^*,y^*)\ne(0,0)$ is a solution of the homogeneous part of degree
$n$ of $F$,
\[\mathcal{H}_n(F(x,y))=F^n(x,y),
\]
that is $F^n(x^*,y^*)=0$.  Equivalently, these are  the zeros of
$\widetilde{F}_1(x,0)$ and $\widetilde{F}_2(y,0)$. In other words,
$[x^*:y^*:0]$ is a point at infinity of $F$ if and only if $x^*/y^*$
is a zero of $\widetilde{F}_1(x,0)=F^n(x,1)$ or  $y^*/x^*$ is a zero
of $\widetilde{F}_2(y,0)=F^n(1,y)$.

Let $\Omega\subset\R^2$ be an unbounded open subset with boundary $\partial
\Omega$ formed by finitely many algebraic curves. It is clear that this subset can
be extended to $\mathbb{RP}^2$. We will call the adherence of this extension
$\bar\Omega.$ When a point at infinity of $F$ is also in $\bar\Omega$, for short
we will say that is a point at infinite which is also in $\Omega$.

\subsection{Isolated points of families of algebraic curves}

To state our main result we need explicit  conditions to check when  a point of a
real algebraic curve $G(x,y)=0$ is isolated. Recall that it is said that
 a point ${\bf p}\in\R^2$ on the curve is {\it isolated} if there exists an open
 neighborhood ${\mathcal U}$ of ${\bf p},$  such that
\[
 {\mathcal U}\cap\{(x,y)\in\R^2\,:\, G(x,y)=0\}={\bf p}.
\]
 Clearly isolated points are singular points of the curve.
 Next result provides an useful criterion
 to deal with this question.

 \begin{lemma}\label{nnou} Let $G(x,y)$ be a real polynomial.
 Assume that  $(0,0)\in\{G(x,y)=0\}$  and that there are
 natural numbers $p,q$ and $m$, with $\gcd(p,q)=1$, and a polynomial
$G^0$ satisfying
 ${G^0(\varepsilon^pX,\varepsilon^qY)}=\varepsilon^m G^0(X,Y)$,
 and such that for all $\varepsilon>0$,
 \[
 {G(\varepsilon^pX,\varepsilon^qY)}=\varepsilon^m G^0(X,Y)+
 \varepsilon^{m+1}G^1(X,Y,\varepsilon),
 \]
for some polynomial function  $G^1$.  If the only real solution of  $G^0(X,Y)=0$
is
 $(X,Y)=(0,0)$, then  the origin is an isolated point of $G(x,y)=0$.
 \end{lemma}

\begin{proof}  Assume without loss of generality that $G^0\ge0.$
We start proving  that $K:=\{(x,y)\in\R^2\,:\, G^0(x,y)=1 \}$ is a
compact set. Clearly it is closed, so it suffices to prove that it
is bounded. Since $G^0$ is a quasi-homogeneous  polynomial we know
that there  exists a natural number $m_0$ such that $m=m_0pq$ and
$G^0(x,y)=P_{m_0}(x^q,y^p)$, where $P_{m_0}$ is a real homogeneous
polynomial of degree $m_0$. The fact that the only real solution of
the equation $G^0(x,y)=0$ is $x=y=0$ implies that $P_{m_0}$ has not
linear factors when we decompose it as a product of real irreducible
factors. Hence $m_0$ is even and $P_{m_0}(x,y)=\prod_{i=1}^{m_0/2}
(A_i x^2+ B_i x y + C_i y^2)$, with $B_i^2-4A_iC_i<0.$ As a
consequence,
\begin{equation}\label{prodd}
G^{0}(x,y)=\prod_{i=1}^{m_0/2} (A_i x^{2q}+ B_i x^q y^p + C_i
y^{2p}),\quad \mbox{with}\quad B_i^2-4A_iC_i<0.
\end{equation}
Assume, to arrive to a contradiction, that $K$ is unbounded.
Therefore it should exist a sequence $\{(x_n,y_n)\}$, tending to
infinity, and such that $G^0(x_n,y_n)=1$. But this is impossible
because the conditions $B_i^2-4A_iC_i<0,$ $i=1,\ldots,m_0/2$, imply
that all the terms $A_i x_n^{2q}+ B_i x_n^q y_n^p + C_i y_n^{2p}$ in
\eqref{prodd} go to infinity. So $K$ is compact.

Let us prove that $(0,0)$ is an isolated point of
$\{(x,y)\in\R^2\,:\,G(x,y)=0\}.$ Assume, to arrive to a
contradiction, that it is not. Therefore there  exists a sequence of
points $\{(x_n,y_n)\}$, tending to 0 and such that $G(x_n,y_n)=0$
for all $n\in\N$.  Consider $G^0(x_n,y_n)=:(g_n)^m>0.$ It is clear
that $\lim_{n\to\infty}(g_n)^m=0.$ Write
$(x_n,y_n)=((g_n)^pu_n,(g_n)^qv_n)$. Notice that
\[
(g_n)^m=G^0(x_n,y_n)=G^0(g_n^pu_n,g_n^qv_n)=(g_n)^mG^0(u_n,v_n).
\]
Then $G^0(u_n,v_n)=1$ and  $(u_n,v_n)\in K,$ for all $n\in\N$. Therefore, taking a
subsequence if necessary, we can assume that
\begin{equation}\label{limit}
\lim_{n\to\infty}(u_n,v_n)=(u^*,v^*)\in K.
\end{equation}
We have that $0=G(x_n,y_n)=(g_n)^m+(g_n)^{m+1}G^1(u_n,v_n,g_n)$.
Dividing by $(g_n)^m$ we obtain that $1=0+g_nG^1(u_n,v_n,g_n)$, and
passing to the limit we get that $1=0$ which gives the desired
contradiction.

Notice that to prove that $\lim_{n\to\infty}g_nG^1(u_n,v_n,g_n) =0$
we need to know that the sequence $\{(u_n,v_n)\}$ remains bounded
and this fact is a consequence of~\eqref{limit}.
\end{proof}

We remark that the suitable  values $p,q$ and $m$ and the function $G^0$ appearing
in the statement of Lemma~\ref{nnou} are usually found  by using the Newton
diagram associated to $G.$

We also need to introduce a new related concept for families of
curves. Consider a one-parameter family of algebraic curves
$G_b(x,y)=0$, $b\in I,$ also depending polynomially of $b$. Let
$(x_0,y_0)\in\R^2$ be an isolated point of $G_b(x,y)=0$ for all
$b\in I$, we will say that $(x_0,y_0)$ is {\it uniformly isolated}
for the family $G_b(x,y)=0$, $b\in I$ if for each $b\in I$ there
exist neighborhoods $\mathcal V\subset I$ and $\mathcal W\subset
\R^2$, of $b$ and $(x_0,y_0)$ respectively, such that for all
$b\in\mathcal{V}$,
\begin{equation}\label{iii} \{(x,y)\in\R^2\,:\, G_b(x,y)=0\}\cap \mathcal{W}=(x_0,y_0).
\end{equation}

Next example shows a one-parameter family of curves that has the
origin isolated for all $b\in\R$ but it is not uniformly isolated
for $b\in I$, with $0\in I,$
\begin{equation}\label{unif}
G_b(x,y)=(x^2+y^2)(x^2+y^2-b^2)(x-1).
\end{equation}
It is clear that  the origin is an isolated point of
$\{G_b(x,y)=0\}$ for all $b\in\R$, but there is no open neighborhood
$\mathcal{W}$ of $(0,0),$ such that~\eqref{iii} holds for any $b$ in
a neighborhood of $b=0$.

Next result is a version of Lemma~\ref{nnou} for one-parameter
families. In its proof we will use some periodic functions
introduced by Lyapunov in his study
 of the stability of degenerate critical points, see \cite{lia}. Let us recall
 them.

Let $u(\varphi)=\Cs(\varphi)$ and $v(\varphi)=\Sn(\varphi)$ be the solutions of
the Cauchy problem:
\[
u'=-v^{2p-1},\,v' =u^{2q-1},\quad u(0)=\root{2q}\of{1/p}\quad \mbox{ and }\quad
v(0)=0,\] where the prime denotes the derivative with respect to $\varphi$.

Then  $x=\Cs(\varphi)$ and $y=\Sn(\varphi)$ parameterize the algebraic curve $p
x^{2q}+q y^{2p}=1,$ that is  $p\Cs^{2q}(\varphi)+q\Sn^{2p}(\varphi)=1,$ and both
functions are smooth $T_{p,q}$-periodic functions, where
$$
T=T_{p,q}=2p^{-1/2q}q^{-1/2p}
\frac{\Gamma\left(\frac{1}{2p}\right)\Gamma\left(\frac{1}{2q}\right)}
{\Gamma\left(\frac{1}{2p}+\frac{1}{2q}\right)},
$$
and $\Gamma$ denotes the  Gamma function.

 \begin{proposition}\label{nnnou} Let $G_b(x,y)$ be a family of real
 polynomials which also depends polynomially on $b$.
 Assume that  $(0,0)\in\{G_b(x,y)=0\}$  and that there are
 natural numbers $p,q$ and $m$, with $\gcd(p,q)=1$, and a polynomial
$G^0_b$ satisfying
 ${G^0_b(\varepsilon^pX,\varepsilon^qY)}=\varepsilon^m G^0_b(X,Y)$,
 and  such that for all $\varepsilon>0$,
 \[
 {G_b(\varepsilon^pX,\varepsilon^qY)}=\varepsilon^m G^0_b(X,Y)+
 \varepsilon^{m+1}G^1_b(X,Y,\varepsilon),
 \]
for some polynomial function  $G^1_b$.  If for all $b\in I\subset \R$, the only
real solution of $G_b^0(X,Y)=0$ is
 $(X,Y)=(0,0)$, then  the origin is an uniformly isolated
 point of $G_b(x,y)=0$ for all $b\in I$.
 \end{proposition}

\begin{proof}  Assume without loss of generality that $G^0_b\ge0.$
Let us write the function $G_b(x,y)$ using the so-called generalized polar
coordinates,
 \[x=\rho^p
\Cs(\varphi),\,y=\rho^q\Sn(\varphi),\quad \mbox{for}\quad  \rho\in\R^+.\] Then
\begin{align}\label{bbbbb}
G_b(x,y)&=G_b(\rho^p \Cs(\varphi),\rho^q\Sn(\varphi))\nonumber\\
&=\rho^mG_b^0(\Cs(\varphi),\Sn(\varphi))+\rho^{m+1}G_b^1(\Cs(\varphi),\Sn(\varphi),\rho).
\end{align}
Using the same notation that in the proof of Lemma~\ref{nnou}, with the obvious
modifications, we know  from~\eqref{prodd} that
\begin{equation*}
G^{0}_b(\Cs(\varphi),\Sn(\varphi) )=\prod_{i=1}^{m_0/2} (A_i(b) \Cs^{2q}(\varphi)+
B_i(b) \Cs^q(\varphi) \Sn^p(\varphi) + C_i(b) \Sn^{2p}(\varphi)),
\end{equation*}
with all $B_i^2(b)-4A_i(b)C_i(b)<0.$ Therefore, it is not difficult to prove that
there exists two positive continuous functions, $L(b)$ and $U(b)$ such that
\[
0<L(b)\le G_b^0(\Cs(\varphi),\Sn(\varphi))\le U(b),
\]
due to the periodicity of the Lyapunov functions and the
discriminant conditions. Dividing the expression~\eqref{bbbbb} by
$\rho^m$ we obtain that the points of
$\{G_b(x,y)=(0,0)\}\setminus\{(0,0\}$ are given by
\begin{equation}\label{nh}
G_b^0(\Cs(\varphi),\Sn(\varphi))+\rho\, G_b^1(\Cs(\varphi),\Sn(\varphi),\rho)=0.
\end{equation}
Fix a compact neighborhood of $b,$ say $\mathcal{V}\subset I.$ Set
$L=\min_{x\in\mathcal{V}} L(b).$ Then there exists $\delta>0$ such that for any
$||(x,y)||\le \delta$ and any $b\in \mathcal{V}$,
 \[
 |\rho\, G_b^1(\Cs(\varphi),\Sn(\varphi),\rho)|<L/2.\]
 Therefore~\eqref{nh} never holds in this region and
\[ \{(x,y)\in\R^2\,:\, G_b(x,y)=0\}\cap \{(x,y)\in\R^2\,: \, ||(x,y)||<\delta
\}=(0,0),
\]
for all $b\in\mathcal{V}$, as we wanted to prove.
\end{proof}

Notice that, the fact that for all $b\in\R$, the origin
of~\eqref{unif} is isolated simply follows plotting the zero level
set of $G_b$. Alternatively, we can apply Lemma~\ref{nnou} with
$p=1,q=1$ and $m=2$ to prove that the origin is isolated when
$b\ne0$ and with $p=q=1$ and $m=4$ when $b=0.$ In any case,
Proposition~\ref{nnnou} can not be used.

\subsection{The method for controlling the sign}

\begin{proposition}\label{lem-inf}
Let $F_b(x,y)$ be a family of real polynomials  depending also
polynomially on a real parameter $b$ and let $\Omega\subset\R^2$ be
an open connected subset having a boundary $\partial \Omega$ formed
by finitely many algebraic curves. Suppose that there exists an open
interval $I\subset\R$ such that:
\begin{enumerate}[(i)]
\item For some $b_0\in I$, $F_{b_0}(x,y)>0$ on $\Omega\subset\R^2$.

\item For all $b\in I,$  $\triangle^2(F_b)\neq 0.$

\item For all $b\in I,$ all points of $ F_b=0$ at infinity which are also in
$\Omega$ do not depend on $b$ and  are uniformly isolated.

\item  For all $b \in I$, $\{F_b=0\}\cap \partial \Omega=\emptyset.$
\end{enumerate}
Then for all $b\in I$, $F_b(x,y)>0$ on $\Omega.$
\end{proposition}
\begin{proof}  Consider the following set
\[
J:=\{ b\in I\,:\, F_{b}(x,y)>0\quad\mbox{for all}\quad (x,y)\in\Omega \}.
\]
By hypothesis (i), $J\ne\emptyset$ because $b_0\in J$. Consider now $\bar b=\sup
J$. We  want to prove that $\bar b\in \partial I.$ If this is true, arguing
similarly with $\inf J$ the result will follow.

We will prove the result by contradiction. So assume that $\bar b\in I.$

Notice that if  $F_{\bar b}(x,y)$  takes positive and negative values on $\Omega$,
by continuity this would happen for any $b$ near enough to $\bar b$. This is in
contradiction with the fact that $\bar b$ is the supremum of $J.$ Therefore,
either $F_{\bar b}(x,y)\ge0$ or $F_{\bar b}(x,y)>0$ in~$\Omega$.

In the first case it is clear that a point $(x_0,y_0)$ where
$F_{\bar b}(x_0,y_0)=0$ has to be a singular point of the curve
$\{F_{\bar b}(x,y)=0\}$. Therefore, by Corollary~\ref{cccc},
$\triangle^2(F_{\bar b})=0$ which is in contradiction with (ii).

In the second case it should exist a sequence of real numbers $\{b_n\}$, with
$b_n\downarrow \bar b$, and a sequence of points $\{(x_n,y_n)\}\in\Omega$ such
that $\lim_{n\to\infty} F_{b_n}(x_n,y_n)=0$.

If the sequence is bounded, renaming it if necessary, we arrive to a convergent
sequence. Call $(\bar x,\bar y)\in \overline \Omega$ its limit, where $\overline
\Omega$ denotes the adherence of $\Omega$. Then $F_{\bar b}(\bar x,\bar y)=0$. By
hypothesis (iv), the point $(\bar x,\bar y)\not\in\partial \Omega$ and we also
know that $F_{\bar b}(x,y)>0$ on $\Omega$. Therefore we have a contradiction and
the sequence $\{(x_n,y_n)\}$ must be unbounded.

This unbounded sequence  can be considered  in the projective space
$\mathbb{RP}^2$. Then this sequence must converge to a point ${\bf p}$ of $F_{\bar
b}(x,y)=0$ at infinity, which is also in~$\mathcal U.$ Since by hypothesis~(iii)
this point is uniformly isolated, there exists a neighborhood $\mathcal V$ of
${\bar b}$ and an open neighborhood $\mathcal W$ of ${\bf p}$ such that  this
point is the only real point in $\mathbb{RP}^2$ of the homogenization of
$F_{b}(x,y)=0$. This is in contradiction with the fact $ F_{b_n}(x_n,y_n)=0$ for
all $n$, and the result follows.
\end{proof}

\subsection{Control of the sign of~\eqref{Mb651}}\label{cas2}

In this subsection we will prove by using Proposition~\ref{lem-inf},
that for $b\in(0,0.6512)$, the function $M_b$ given in~\eqref{Mb651}
is positive on $\Omega=\R^2$.

To check hypothesis~(i), we  prove  that $M_{1/2}>0$ for all $\R^2.$
For this value,
$${\textstyle M_{1/2}=\frac{15}{2}\,{x}^{4}{y}^{2}-{\frac{21}{4}}\,{x}^{3}{y}^{3}
+\frac{21}{2}\,{x}^{2}{y}^ {4}-{\frac {123}{16}}\,{x}^{2}{y}^{2}+{\frac
{21}{16}}\,x {y}^{3}+\frac{5}{2}\,{x}^{4}-{\frac {7}{16}}\,{x}^{2}+{\frac
{15}{64}}\,{y}^{2}+{\frac {13} {64}}.}
$$
We think $M_{1/2}$ as a polynomial in $x$  and $y$ as a parameter
and we apply Lemma~\ref{ll}. If $y=0$ then $M_{1/2}$ reduces to the
polynomial $(5/2)x^4-(7/16)x^2+13/64$ which is positive on $\R$.
Now, we compute $\triangle_x(M_{1/2})$ and we obtain a polynomial in
the variable $y$ of degree 20. By using the Sturm method it is easy
to see  that it does not have real roots. Moreover, the coefficient
of $x^4$ is $5(3y^2+1)/2>0.$  Therefore,  $M_{1/2}>0$ on $\R^2$, as
we wanted to see.

\smallskip

To check hypothesis (ii) we compute the double discriminant of $M_b$
and we obtain that
 $\triangle_{x,y}^2(M_b)$ is a
polynomial in $b$ of degree 1028, of the following form
\begin{align*}
\triangle_{x,y}^2
(M_b)=&b^{320}(b^2-2)^{40}(3b^2-2)^5(3b^2-4)(2b^6-4b^4-3b^2+2)\times
 \\&\times(b^6-2b^4-3b^2+2)
(P_2(b^2))^8(P_{6}(b^2))^4(P_{32}(b^2))^2(P_{33}(b^2))^6,
\end{align*}
where  $P_i$  are polynomials of degree $i$ with rational
coefficients. By using the Sturm method we localize  the real roots
of each factor of $\triangle_{x,y}^2(M_b)$ and we obtain that in the
interval $(0,0.6512)$ none of them has real roots. In fact
$P_{32}(b^2)$ has a root in $(0.6513,0.6514)$ and that is the reason
for which we can not increase more the value of $b.$ Therefore
$\triangle_{x,y}^2(M_b)\neq 0$ for all $b\in(0,0.6512)$.

Finally we have to check hypothesis (iii). Notice that in this case $\partial
\Omega=\emptyset$ and so (iv) follows directly.

The zeros at infinity are given by the directions
\[
\mathcal{H}_6(M_b)=6x^2y^2[(2-3b^2)x^2-2b^2(2-b^2)xy+(2-b^2)y^2]=0.
\]
For $|b|<0.7275$ it has only the non-trivial solutions $x=0$ and
$y=0$. The homogenization of $M_b$ is
\begin{equation}\label{MbH651}
\begin{array}{lll}
\widetilde{M}_b&=&6[(2-3b^2)x^4y^2-2b^2(2-b^2)x^3y^3+(2-b^2)x^2y^4]
+2(2-3b^2)x^4z^2\\
&&-3b^2(14-15b^2)x^2y^2z^2+12b^4(2-b^2)xy^3z^2-b^2(4-9b^2)x^2z^4\\
&&+3b^4(2-3b^2)y^2z^4+b^4(4-3b^2)z^6,
\end{array}
\end{equation}
and hypothesis (iii) is equivalent to prove that $(0,0)$ is an
uniformly isolated singularity for
$\widetilde{M}^1_{b}(x,z)=\widetilde{M}_b(x,1,z)$ and that $(0,0)$
is also an uniformly isolated singularity for
$\widetilde{M}^2_{b}(y,z)=\widetilde{M}_b(1,y,z)$.

First we prove this result  for $\widetilde{M}_b^1(x,z)$. From
\eqref{MbH651},
$$
\begin{array}{lll}
\widetilde{M}^1_{b}(x,z)&=&6[(2-3b^2)x^4-2b^2(2-b^2)x^3+(2-b^2)x^2]
+2(2-3b^2)x^4z^2\\
&&-3b^2(14-15b^2)x^2z^2+12b^4(2-b^2)xz^2-b^2(4-9b^2)x^2z^4\\
&&+3b^4(2-3b^2)z^4+b^4(4-3b^2)z^6.
\end{array}
$$
Hence,
$$\widetilde{M}^1_{b}(\varepsilon^2X,\varepsilon Z)=\Big(6(2-b^2)X^2 +12b^4(2-b^2)XZ^2
 +3b^4(2-3b^2)Z^4\Big) \varepsilon^4+O(\varepsilon^5).
$$
The discriminant with respect to $X$ of the homogeneous polynomial
$T(X,W):=6(2-b^2)X^2 +12b^4(2-b^2)XW  +3b^4(2-3b^2)W^2,$ where
$W=Z^2,$ is
\[
\triangle_X(T)=72W^2b^4(b^2-2)(2b^6-4b^4-3b^2+2).
\]
Since its smallest positive root is greater than $0.673$ it holds for
$b\in(0,673)$  that  $T(X,W)=0$ if and only if $(X,W)=(0,0).$
 Therefore by Proposition~\ref{nnnou} the point $(0,0)$ is an uniformly
 isolated point of the
curve~$\widetilde{M}^1_{b}(x,z)=0,$ for these values of $b$.

For the other point, since
$$
\begin{array}{lll}
\widetilde{M}^2_{b}(y,z)&=&6[(2-b^2)y^4-2b^2(2-b^2)y^3+(2-3b^2)y^2]
+2(2-3b^2)z^2\\
&&-3b^2(14-15b^2)y^2z^2+12b^4(2-b^2)y^3z^2-b^2(4-9b^2)z^4\\
&&+3b^4(2-3b^2)y^2z^4+b^4(4-3b^2)z^6,
\end{array}
$$
we have that
$$\widetilde{M}^2_{b}(\varepsilon Y,\varepsilon Z)=2(2-3b^2)\Big(3Y^2+Z^2\Big) \varepsilon^2+O(\varepsilon^3),
$$
and the result follows for $b\in(0,\sqrt{2/3})\approx(0,0.816),$ by applying again
the same proposition.

So, we have shown that for $b\in(0,0.6512)$ all the hypotheses of
the Proposition~\ref{lem-inf} hold. Therefore  we have proved that
for $b\in(0,0.651]$, $M_b(x,y)>0$ for all $(x,y)\in\R^2$.

\subsection{Control of the sign of~\eqref{ap1}}\label{cas3}

 The
numerator of the function $M_b$ given in \eqref{ap1} is a polynomial of the following form
\begin{equation}
\begin{array}{ll}
N_b(x,y)=f_0(x,b)+f_1(x,b)y+f_2(x,b)y^2+f_3(x,b)y^3+f_4(x,b)y^4,
\end{array}\label{Mb817}
\end{equation}
where {\small
$$
\begin{array}{lll}
f_0(x,b)&=&
90b^{36}x^{10}-15b^{18}(6b^{20}-5)x^8+15b^{18}(24b^4-59b^2+24)x^6\\&&
-(378b^{24}-810b^{22}+360b^{20}-300b^4+675b^2-300)x^4
\\&&-15b^2(18b^{22}-24b^{20}+21b^4-45b^2+20)x^2-75b^4(-4+3b^2),
\end{array}
$$
$$
\begin{array}{lll}
f_1(x,b)&=&180b^{36}x^7+12b^{18}(60b^{16}+50b^{14}+18b^{10}+25)x^5
-20b^{10}(36b^{12}\\&&-54b^{10}+54b^8-30b^6-25b^4-9)x^3-180b^{20}(3b^2-4)x,
\end{array}
$$
$$
\begin{array}{lll}
f_2(x,b)&=&270b^{36}x^{10}-45b^{18}(6b^{20}+2b^{18}-5)x^8+3b^{18}
(30b^{20}+120b^{16}\\&&+100b^{14}-90b^{12}+36b^{10}+360b^4-615b^2+335)x^6
-(360b^{36}\\&&+300b^{34}+108b^{30}+2214b^{24}
-3690b^{22}+3435b^{20}+360b^{18}
\\&&-300b^{16}-250b^{14}+225b^{12}-90b^{10}-900b^4+1350b^2-900)x^4
\\&&-b^2(468b^{22}-540b^{20}-1080b^{18}+300b^{16}+250b^{14}+90b^{10}
\\&&+1845b^4-3075b^2+2475)x^2-90b^4(4b^2-5),
\end{array}
$$
$$
\begin{array}{lll}
f_3(x,b)&=&-180b^{20}(b^{10}-3)x^7+30b^2(6b^{34}+6b^{30}-24b^{22}+18b^{20}
-72b^{18}\\&&-5b^{10}+15)x^5+30b^2(24b^{24}-36b^{22}+72b^{20}+10b^{16}+5b^{12}
\\&&-20b^4+15b^2-60)x^3-20b^4(36b^{18}-54b^{16}+54b^{14}+30b^{12}
\\&&+25b^{10}+9b^6-30b^4+45b^2-90)x,
\end{array}
$$
$$
\begin{array}{lll}
f_4(x,b)&=&90b^{36}x^8-3b^{18}(30b^{20}+120b^{16}+100b^{14}+36b^{10}-25)x^6
\\&&+b^{10}(360b^{26}+300b^{24}+198b^{20}+360b^{12}-615b^{10}+720b^8
\\&&-300b^6-250b^4-90)x^4+(-738b^{24}+1080b^{22}-1080b^{20}
\\&&+300b^{18}+250b^{16}+315b^{12}+300b^4-450b^2+900)x^2+15b^6.
\end{array}
$$}

We will prove that $N_b\ge0$  on $\Omega:=\{(x,y)\,:\, xy+1>0\}$ for all
$b\in(0,0.817]$ and if it vanishes this only happens at some isolated points. We
will use again Proposition~\ref{lem-inf}. Notice that $\partial \Omega
=\{(x,y)\,:\, xy+1=0\}.$

\begin{figure}[h]
\centering\epsfig{file=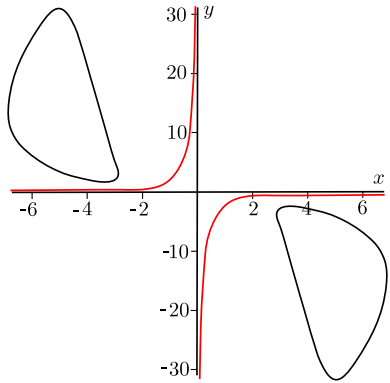,width=5cm} \caption{Curves
$N_b=0$ and $xy+1=0$ with $b=0.817$.} \label{M817Hip}
\end{figure}

It is not difficult to verify that $\{N_b(x,y)=0\}\cap\{xy+1=0\}=\emptyset$ for
$b\in(0,0.8171)$, see Figure~\ref{M817Hip}. It suffices to see that for these
values of $b$, and $x\ne0,$ the one variable function $N_b(x,1/x)$, never
vanishes. We skip the details. Therefore hypothesis (iv) is satisfied.

\smallskip
For proving that hypothesis (ii) of Proposition~\ref{lem-inf} holds  we compute
the double discriminant $\triangle_{y,x}^2(N_b)$. It is an even polynomial in $b$,
of degree 21852, of the following form
\begin{equation}{\small\label{DDMb}
b^{7566}(3b^2-4)(159b^4-380b^2+225)^2(P_{71}(b^2))^2(P_{386}(b^2))^4(P_{587}(b^2))^6
(P_{965}(b^2))^2,}
\end{equation}
where  $P_i$ are polynomials of degree $i$ with rational coefficients.
 By using the Sturm method it is
easy to see that its first 4 factors  do not have  real roots in
$(0,0.8171)$. We replace $b^2=t$ in the next three polynomials  to
reduce their degrees and we obtain $\mathcal{P}_1(t):=P_{386}(t)$,
$\mathcal{P}_2(t):=P_{587}(t)$, and $\mathcal{P}_3(t):=P_{965}(t).$
It suffices to study their number of  real roots in $(0,0.6678]$,
because $0.6678>(0.8171)^2$. Our computers have no enough capacity
to get their Sturm sequences. Therefore  we will use the Descartes
approach as it is explained in Appendix I.

We consider first the polynomial $\mathcal{P}_1(t)$. Its normalized version
$N^{0.68}_0(\mathcal{P}_1)$ has all their coefficients positive. Therefore
$\mathcal{P}_1(t)$ has no real roots in $(0,0.68)$ as we wanted to see.

Applying the Descartes rule to the normalized versions of
$\mathcal{P}_2(t)$, $ N^{0.561}_0(\mathcal{P}_2),$
$N_{0.561}^{0.811}(\mathcal{P}_2)$ and
$N_{0.562}^{0.812}(\mathcal{P}_2), $ we obtain that the number of
zeros in the intervals $(0,0.561),$ $(0.561,0.811)$ and
$(0.562,0.812)$ is 0, 1 and 0 respectively. That is, there is only
one root of $\mathcal{P}_2(t)$  in $(0,0.812)$, it is simple  and it
belongs to $(0.561,0.562)$. Refining this interval with Bolzano
Theorem we  prove that the root is in the interval
$(0.5617,0.5618)$.

Finally to study $\mathcal{P}_3(t)$ we consider $N^{11/20}_0(\mathcal{P}_3),$
$N^{7/12}_{11/20}(\mathcal{P}_3)$ and $N^{52/75}_{7/12}(\mathcal{P}_3).$ By
Descartes rule  we obtain that the number of zeros of $\mathcal{P}_3$ in the
corresponding intervals is 0, 1 and 1 or 3, respectively. By Bolzano Theorem we
can localize more precisely these zeros and prove that in the last interval there
are exactly 3 zeros. So we have proved that the polynomial $\mathcal{P}_3$ has
exactly 4 zeros in the interval $(0,{52}/{75})\approx(0,0.693)$, and each one of
them is contained in  one of the following intervals
\[
\left({0.5614},{0.5615}\right),\, \left({0.6678},{0.6679}\right),\,
\left({0.6690},{0.6700}\right),\, \left({0.6870},{0.6880}\right).
\]
In brief, for $t\in(0,0.6678]$ the double discriminant
$\triangle_{y,x}(N_b)$ only vanishes at two points $t=t_1$ and
$t=t_2$ with $t_1\in(0.5614,0.5615)$ and $t_2\in(0.5617,0.5618)$.
Therefore we are under the hypothesis (ii) of
Proposition~\ref{lem-inf} for $b$ belonging to each of  the
intervals $(0,b_1)$, $(b_1,b_2)$ and $(b_2,0.8171),$ where \[
b_1:=\sqrt{t_1}\approx 0.749301 ,\quad
b_2:=\sqrt{t_2}\approx0.749478.
\]

To ensure that on each interval we are under the hypotheses (i) of the proposition
we prove  that   $N_b$ does not vanish on $\Omega$  for one value of $b$ in each
of the above three intervals. We take
\[
\frac12\in(0,b_1),\quad \frac{7494}{10000}\in(b_1,b_2),\quad \mbox{and}\quad
\frac34\in(b_2,0.8171).
\]
We study with detail the case  $b=1/2.$ The other two cases can be treated
similarly and we skip the details. So we have to study on $\Omega$ the sign of the
function {\footnotesize
$$
\begin{array}{lll}
N_{1/2}&=& {\frac {135}{34359738368}}{x}^{10}{y}^{2}+{\frac
{45}{34359738368}}{x}^{8}{y}^{4}+{\frac {45}{34359738368}}{x}^{10}+{\frac
{117964485 }{137438953472}}{x}^{8}{y}^{2}\\\\&&+{\frac
{138195}{268435456}}{x}^{7}{ y}^{3}+{\frac
{39253779}{137438953472}}{x}^{6}{y}^{4}+{\frac {
39321555}{137438953472}}{x}^{8}+{\frac {45}{17179869184}}{x}^{7}y\\\\&&+ {\frac
{320504301}{137438953472}}{x}^{6}{y}^{2}+{\frac {
1932072223485}{17179869184}}{x}^{5}{y}^{3}-{\frac {906074381}{
8589934592}}{x}^{4}{y}^{4}+{\frac {645}{1048576}}{x}^{6}\\\\&&+{\frac {
1229859}{1073741824}}{x}^{5}y+{\frac {5315442024413}{8589934592}}{
x}^{4}{y}^{2}-{\frac {1808748465}{4194304}}{x}^{3}{y}^{3}+{\frac {
6763995071}{8388608}}{x}^{2}{y}^{4}\\\\&&+{\frac {1258289751}{8388608}}{
x}^{4}+{\frac {55625}{262144}}{x}^{3}y-{\frac {1910154937}{4194304}}
{x}^{2}{y}^{2}+{\frac {26361865}{262144}}x{y}^{3}+{\frac {15}{64}}
{y}^{4}\\\\&&-{\frac {316538295}{8388608}}{x}^{2}+{\frac {585}{1048576}} xy+{\frac
{45}{2}}{y}^{2}+{\frac {975}{64}}.
\end{array}
$$}
\!\!We consider $N_{1/2}$ as a polynomial in $x$ with coefficients in $\R[y]$ and
we apply Lemma~\ref{ll} with $\Omega_y=(-1/y,\infty)$ when $y>0$ and
$\Omega_0=(-\infty,\infty)$. Notice that for the symmetry of the function there is
no need to study   the zone $y<0$ because $N_{1/2}(-x,-y)=N_{1/2}(x,y)$. We
introduce the following notation $S_{y}(x):=N_{1/2}(x,y)$. We prove the following
facts:

\begin{enumerate}[(i)]

\item If we write $S_y(x)=\sum_{i=1}^10 s_{i}(y)x^{i},$ then
$s_{10}(y)=k(1+3y^2)$ for some $k\in\Q^+$. Therefore $s_{10}(y)>0$ for all
$y\in\R.$

\item If $y=0$ then $S_0(x)$ is an even polynomial of degree 10 and it is easy to
see that  $S_0(x)>0$ over $\R$.

\item We already know that $\{S_y(x)=0\}\cap \partial
\Omega=\emptyset.$

\item Some computations give that
\[
\triangle_x (S_y)=P_{35}(y^2),
\]
where $P_{35}$ is a polynomial of degree 35. Moreover, using once more the Sturm
method, we get that $P_{35}(y^2)$ has only two positive roots $0<y_1<y_2$, with
$y_1\approx 0.588423$ and $y_2\approx 6065.2946$. From this result it is easy to
prove that:
\begin{enumerate}[(a)]
\item If $y\in[0,y_1)\cup(y_2,\infty)$, then $S_y(x)>0.$
\item If $y\in(y_1,y_2)$, then $S_y(x)$ has only two real roots, say $x_1(y)<x_2(y)$, and
none of them belongs to the interval $(-1/y,\infty)$. So  $S_y(x)>0$
on $(-1/y,\infty)$.
\item If $y\in\{y_1,y_2\}$, then $S_y(x)$ has only a real root,   $x_1(y)$,
which is a double root and $x_1(y)\not\in(-1/y,\infty)$. So, again  $S_y(x)>0$ on
$(-1/y,\infty)$.

\end{enumerate}

\end{enumerate}

Thus, by  Lemma~\ref{ll}, the function $N_{1/2}$ is positive on $(x,y)\in\Omega$,
as we wanted to see. In fact, its level curves are like the ones showed in
Figure~\ref{M817Hip}. The straight lines $y=y_1$ and $y=y_2$ correspond to the
lower and upper tangents to the oval contained in the second quadrant.

To be under all the hypotheses of Proposition~\ref{lem-inf} it only
remains to study the function $\widetilde{N}_b$ at infinity. We
denote by $\widetilde{N}_b(x,y,z)$ its homogenization in
$\mathbb{RP}^2$ and by $\widetilde{N}^1_{b}(x,z)$ and
$\widetilde{N}^2_{b}(y,z)$ the expressions of the function
$\widetilde{N}_b$ in the planes $\{(x,z)\}$ and $\{(y,z)\}$,
respectively. Since $
\mathcal{H}_{12}({N}_b)=90b^{36}x^8y^2[3x^2+y^2], $ the only
non-trivial solutions of $\mathcal{H}_{12}({N}_b)=0$  are  $x=0$ and
$y=0$. Hence these directions give rise to two  points of $N_b$ at
infinity which are also on the region~$\Omega$. They correspond to
the points $(0,0)$ of the algebraic curves
$\widetilde{N}^1_{b}(x,z)=0$ and  $\widetilde{N}^2_{b}(y,z)=0$. We
have to prove that both points are uniformly isolated.

\smallskip

Similarly that in the previous subsection, we write
{\footnotesize\begin{align*}
 \widetilde{N}^1_{b}&(\varepsilon X,\varepsilon
Z)=\Big(90b^{36}X^8-3b^{18}(30b^{20}+120b^{16}+100b^{14}+36b^{10}-25)X^6Z^2\\
&+b^{10}(360b^{26}+300b^{24}+198b^{20}+360b^{12}-615b^{10}+720b^8-300b^6
-250b^4-90)X^4Z^4\\
&+(-738b^{24}+1080b^{22}-1080b^{20}+300b^{18}+250b^{16}
+315b^{12}+300b^4-450b^2+900)X^2Z^6\\
&+15b^6Z^8 \Big) \varepsilon^8+O(\varepsilon^9)
\end{align*}}
and
\[\widetilde{N}^2_{b}(\varepsilon
Y,\varepsilon Z)=90b^{36} (3Y^2+Z^2)\varepsilon^2+O(\varepsilon^3).\]

By Proposition~\ref{nnnou}, for  the second algebraic curve it is
clear that for all $b>0$ $(0,0)$ is an isolated point.

For studying the first one we denote by $R(X,Z)$ the homogenous polynomial
accompanying  $\varepsilon^8$ and we obtain that
\[
\triangle_X(R(X,Z))=Z^{56}b^{150}(P_{71}(b^2))^2,
\]
for some polynomial $P_{71}$ of degree $71$ and integer
coefficients. Since the smallest positive root of this  polynomial
is greater that $0.92$ we can easily prove that for $b<0.92,$
$R(X,Z)=0$ if and only if $X=Z=0.$ Therefore we can use again
Proposition~\ref{nnnou} and  prove that $(0,0)$ is an uniformly
isolated point of the curve for these values of $b$.

So, if we write
\[
(0,0.8171)=(0,b_1)\cup\{b_1\}\cup(b_1,b_2)\cup\{b_2\}\cup(b_2,0.8171),
\]
we can apply Proposition~\ref{lem-inf} to each one of the open
intervals to prove that for $b\in(0,0.817]\setminus\{b_1,b_2\}$ it
holds that $N_b(x,y)>0$ for all $(x,y)$ in $\Omega$. By continuity,
for the two values  $b\in\{b_1,b_2\}$, we obtain that
$N_b(x,y)\ge0$. Since $\triangle_y(N_b)\not\equiv0$ either it is
always positive or it vanishes only at some isolated points, as we
wanted to prove.

It can be seen that for $b \gtrsim \hat b\approx 0.81722$,
$N_b(x,y)$  changes sign on $\Omega$ because there appears one oval
in the set $\{N_b(x,y)=0\}$. The value $\hat b^2\approx 0.6678492$
corresponds to the root of $\mathcal{P}_3$ in the interval
$\left({0.6678},{0.6679}\right)$ that has appeared in the proof as a
root of the double discriminant.

\subsection*{Acknowledgements}
The first two authors are partially supported by a MCYT/FEDER grant number
MTM2008-03437 and by a CIRIT grant number 2009SGR 410.

\end{document}